\documentclass[12pt]{article}        

 \usepackage[T1]{fontenc}             
 \usepackage[width=16cm, height=22cm]{geometry}   

 \usepackage{mathtools} 
 \usepackage{amssymb}          
\usepackage{hyperref} 
 \usepackage[amsmath,thmmarks,hyperref]{ntheorem}

 \usepackage{icomma}                   
 \usepackage{enumerate}          
 \usepackage{color}                 
 \usepackage{setspace}                 
 \usepackage{needspace}               
 \usepackage{url}                    
 \usepackage{mathrsfs}                
 \usepackage{esint}    
 \usepackage{slashed}
 \usepackage{verbatim}

 \numberwithin{equation}{section}
 
 \usepackage[numbers,square]{natbib}
 \usepackage{tikz}
\usetikzlibrary{arrows.meta}
\usetikzlibrary{cd}

\newcounter{comments}
\newenvironment{displaycomment}{\begin{list}{}{\rightmargin=0cm\leftmargin=1cm}\item\sf\begin{footnotesize}\color{blue}}{\end{footnotesize}\end{list}}

\allowdisplaybreaks

\theoremstyle{nonumberplain}  
\theoremheaderfont{\itshape}  
\theorembodyfont{\normalfont}  
\theoremseparator{.}  
\theoremsymbol{$\Box$}
\newtheorem{proof}{Proof} 
  
\theoremstyle{plain}  
\theoremheaderfont{\normalfont\bfseries}  
\theorembodyfont{\itshape}  
\theoremsymbol{~}
\theoremseparator{.}  
\newtheorem{proposition}{Proposition}[section]  
\newtheorem{corollary}[proposition]{Corollary}  
  
\newtheorem{lemma}[proposition]{Lemma}  
  
\newtheorem{theorem}[proposition]{Theorem}   

\theorembodyfont{\normalfont}  
 
\newtheorem{remark}[proposition]{Remark}

\newtheorem{definition}[proposition]{Definition}

\theoremstyle{nonumberplain}
\theoremheaderfont{\normalfont\bfseries}  
\theorembodyfont{\itshape}  
\theoremsymbol{~}
\theoremseparator{.}

\newcommand{\R}{\mathbb{R}}
\newcommand{\V}{\mathcal{V}}
\newcommand{\N}{\mathbb{N}}
\newcommand{\Z}{\mathbb{Z}}
\newcommand{\C}{\mathbb{C}}
\newcommand{\dd}{\mathrm{d}}

\newcommand{\tr}{\mathrm{tr}}
\newcommand{\End}{\mathrm{End}}

\newcommand{\Ber}{\mathsf{B}}
\newcommand{\ber}{\mathsf{b}}

\newcommand{\cc}{\mathbf{c}}

\DeclareMathOperator{\str}{\mathrm{str}}

\newcommand{\id}{\mathrm{id}}

\renewcommand{\tilde}{\widetilde}

\newcommand{\T}{\mathbb{T}}

\renewcommand{\L}{\text{\normalfont\sffamily L}}

\newcommand{\pf}{\mathrm{pf}}
\newcommand{\Pf}{\mathrm{Pf}}
\newcommand{\sign}{\mathrm{sign}}
\newcommand{\sgn}{\mathrm{sgn}}
\newcommand{\Spin}{\mathrm{Spin}}

\newcommand{\Cl}{\mathrm{Cl} }
\renewcommand{\top}{\mathrm{top}}

\newcommand{\defeq}{\stackrel{\mathrm{def}}{=}}

\newcommand{\NZ}[1]{\marginpar{#1}}

\title{The Fermionic integral on loop space and the {P}faffian line bundle}
\author{Florian Hanisch\footnote{Universit\"at Potsdam. fhanisch@uni-potsdam.de} ~and Matthias Ludewig\footnote{Universit\"at Regensburg, matthias.ludewig@mathematik.uni-regensburg.de}}

\begin{document}

\maketitle

\begin{abstract}
 As the loop space of a Riemannian manifold is infinite-dimensional, it is a non-trivial problem to make sense of the ``top degree component'' of a differential form on it. In this paper, we show that a formula from finite dimensions generalizes to assign a sensible ``top degree component'' to certain composite forms, obtained by wedging with the exponential (in the exterior algebra) of the canonical 2-form on the loop space. The result is a section on the Pfaffian line bundle on the loop space. We then identify this with a section of the line bundle obtained by transgression of the spin lifting gerbe. These results are a crucial ingredient for defining the fermionic part of the supersymmetric path integral on the loop space.
\end{abstract}

\section{Introduction}

An important part in the task of understanding the geometry and topology of the loop space $\L X$ of a  finite-dimensional manifold $X$ is the study of its differential forms. 
In finite dimensions, a key feature of differential forms is that they can be {\em integrated}, which gives a linear functional on the space of differential forms.
One of the fundamental properties of this integration functional is that it is only non-zero on forms of top degree; at first glance, it therefore seems hopeless to define integration in an infinite-dimensional situation, as there are differential forms of arbitrarily high degree.

Indeed, there is no reasonable way to define the top degree component $[\theta]_{\mathrm{top}}$ of a general differential form $\theta \in \Omega(\L X)$. However, in this paper we show that if we fix a Riemannian metric on $X$ and define the canonical two form on the loop space by\footnote{Throughout, we write $\T = S^1 = \R/\Z$ and loops are denoted by $\gamma: \R \rightarrow X$.}
\begin{equation} \label{CanonicalTwoForm}
  \omega (V, W) \defeq \int_{\T} \bigl\langle V(t), \nabla_{\dot{\gamma}} W(t)\bigr\rangle\, \dd t,
\end{equation}
for $V, W \in T_\gamma \L X = C^\infty(\T, \gamma^*TX)$, then there is a natural way to make sense of the top-degree component of the {\em composite form} $e^{-\omega} \wedge \theta$ for suitable forms $\theta$, using analogies to the finite-dimensional situation.   Here $e^{-\omega}$ denotes the exponential of ${-\omega}$ in the algebra $\Omega(\L X)$. As we did not fix an ``orientation'' on the loop space, this top degree component $[e^{-\omega} \wedge \theta]_{\mathrm{top}}$ will be a section of the Pfaffian line bundle on $\L X$. 

Our starting point is the Pfaffian formula
 \begin{equation}  \label{TopDegreeFiniteDimensionsIntro}
    \bigl[e^{{-\omega}} \wedge \vartheta_1 \wedge \cdots \wedge \vartheta_N\bigr]_{\mathrm{top}} = \pf\Bigl( \langle \vartheta_i, A^{-1} \vartheta_j \rangle \Bigr)_{1 \leq i, j \leq N} \cdot \pf(A),
  \end{equation}
 which holds for any invertible skew-symmetric matrix $A$ with induced 2-form $\omega[v, w] = \langle v, A w\rangle$ on a finite-dimensional vector space $V$ with scalar product, and any collection $\vartheta_1, \dots, \vartheta_N \in V^\prime$ (there is a more complicated formula in case that $A$ has a non-trivial kernel, see Prop.~\ref{Prop:FermionicWick} below). Without the datum of an orientation, the Pfaffian $\pf(A)$ is not a number but an element of the line $\Lambda^{\mathrm{top}} V^\prime$, the top exterior power of the dual  $V^\prime$. 
 
Applying formula \eqref{TopDegreeFiniteDimensionsIntro} to the infinite-dimensional situation by analogy, where with a view on \eqref{CanonicalTwoForm}, we have $A = \nabla_{\dot{\gamma}}$, the covariant derivative along a loop $\gamma \in \L X$, it turns out that $A^{-1} = \nabla_{\dot{\gamma}}^{-1}$ is a well-defined bounded operator, so the scalar products $\langle \theta_i, A^{-1} \theta_j \rangle_{L^2}$ for (square-integrable) 1-forms $\theta_1, \dots, \theta_N$ are well-defined. The Pfaffian of $\nabla_{\dot{\gamma}}$ on the other hand lives in the \emph{Pfaffian line} $\Pf_\gamma$. Extending linearly, this gives a canonical interpretation of the top degree component of $e^{-\omega} \wedge \theta$ for any form $\theta$ that is a sum of wedge product of (square-integrable) one-forms. This top degree is a section of the Pfaffian line bundle, which we also denote by $[e^{-\omega} \wedge \theta]_{\mathrm{top}}$, by abuse of notation.

 A spin structure on $X$ provides a trivialization of the Pfaffian line bundle \cite{PratWaldron, StolzTeichner, Waldorf1}, thus turning the top degree of a differential form into a \emph{function} on  $LX$. The following result derives an explicit formula for this function in terms of spin geometry and is the main result of this paper.
 
 \begin{theorem} \label{MainThm1Intro}
Let $X$ be a spin manifold. Then under the canonical trivialization of the Pfaffian line bundle provided by the spin structure, we have the formula
\begin{equation} \label{FormulaqIntro}
[e^{-\omega} \wedge \theta_1 \wedge \cdots \wedge \theta_N]_{\mathrm{top}}  =  {2^{-N/2}}\sum_{\sigma \in S_N} \mathrm{sgn}(\sigma) \int_{\Delta_N} \str \left([\gamma\|_{\tau_1}^0]^\Sigma \prod_{a=1}^N \cc\bigl(\theta_{\sigma_a}(\tau_a)\bigr) [\gamma\|_{\tau_{a+1}}^{\tau_a}]^\Sigma\right)\dd \tau.
\end{equation}
for the top degree component, initially defined using \eqref{TopDegreeFiniteDimensionsIntro} respectively its generalization to non-invertible $A$.
Here $[\gamma\|_\bullet^\bullet]^\Sigma$ denotes parallel transport in the spinor bundle along the loop $\gamma$, $\cc$ denotes Clifford multiplication, $S_N$ denotes the $N$-th symmetric group and
  \begin{equation*}
    \Delta_N \defeq \{\tau = (\tau_N, \dots, \tau_1) \mid 0 \leq \tau_1 \leq \dots \leq \tau_N \leq 1\}
  \end{equation*}
  is the standard simplex. 
\end{theorem}

Formula \eqref{FormulaqIntro} is very much in the spirit of the Hamiltonian-Lagrangian correspondence, as explained by Bismut in \S3.12 of \cite{BismutDH}. We find this result quite remarkable, as the original definition \eqref{TopDegreeFiniteDimensionsIntro} in terms of Pfaffians makes no reference to spin geometry or parallel transport whatsoever. 

If $X$ is not necessarily spin, the right hand side of \eqref{FormulaqIntro} still has a canonical interpretation as an element $\Ber(\theta_N \wedge \cdots \wedge \theta_1)$ of the line $\mathcal{L}_\gamma$, where $\mathcal{L}$ is the line bundle over $\L X$, obtained by transgression of the spin lifting gerbe. This line bundle is canonically isomorphic to the Pfaffian line, and we prove a generalization of Thm.~\ref{MainThm1Intro} which states that the isomorphism maps these canonical sections to each other (see Thm.~\ref{ThmMainTheoremLvsPf}). This generalizes a result of Prat-Waldron \cite[Thm.~B]{PratWaldron}. 

\medskip

In case that $X$ is a compact spin manifold, the right hand side of \eqref{FormulaqIntro} may be integrated with respect to the Wiener measure, thereby making sense of the loop space differential form integral of $e^{-\omega} \wedge \theta$ for a wide class of integrands $\theta$. This is the path integral of the $\mathcal{N}=1/2$ supersymmetric $\sigma$-model, which will be discussed in detail in our paper \cite{HanischLudewig1}. Together with \cite{HanischLudewig1}, the present paper clarifies the connection between this path integral and the loop space Chern character constructed in \cite{GueneysuLudewig}.

\medskip

Fermionic integrals, such as the one considered in this paper, have been studied extensively in mathematics and theoretical physics. Indeed, formula \eqref{TopDegreeFiniteDimensionsIntro} closely resembles the construction of the Berezin integral on finite-dimensional supermanifolds, where superfunctions can be viewed as section of the exterior power of a certain vector bundle. Integration is then defined by first projecting out the top degree part of the integrand and then performing an ordinary integral of the resulting function with respect to some volume measure. Moreover, Pfaffians have been used by many authors to define the Fermionic analog of a Gaussian integral, see e.g., \cite{Lott}, \cite{MR1689252} and in particular \cite{MR1905424} as well as references therein. Aiming at a rigorous construction of certain physical field theories, \cite{MR1905424} also constructs infinite-dimensional Grassmann integrals through approximation by finite-dimensional Pfaffians and it should be interesting to relate our approach to theirs.

\medskip

The plan of the paper is as follows. After introducing some preliminaries on loop space differential forms and spin geometry needed in the sequel, we prove a formula (to our knowledge originally by Atiyah) relating the zeta-regularized determinant to the parallel transport in the spinor bundle. This is needed  in the sequel and gives the special case $N=0$ in Thm.~\ref{MainThm1Intro}. In \S\ref{SectionLineBundles}, we introduce the Pfaffian line bundle $\Pf$ and the Spin line bundle $\mathcal{L}$ over $\L X$ and construct the canonical isomorphism between the two in a language suited for the calculations to come. Then in \S\ref{SectionTopDegree}, we can finally construct the top degree map as an element of the Pfaffian line, the image of which under the previously constructed canonical isomorphism is calculated in \S\ref{SectionMainTheorem}.

\paragraph{Acknowledgements.} 
We are indebted to Batu G\"uneysu and Achim Krause for helpful discussions.
We thank the Max-Planck-Institute for Gravitational Physics in Potsdam (Albert-Einstein-Institute), the Max-Planck-Institute for Mathematics in Bonn, the Institute for Mathematics at the University of Potsdam and the University of Adelaide for hospitality and financial support.  The second-named author was supported by the Max-Planck-Foundation and the ARC Discovery Project grant FL170100020 under Chief Investigator and Australian Laureate Fellow Mathai Varghese.

\tableofcontents

\subsection{Preliminaries I: Loop Space Differential Forms} \label{SectionLoopSpaceForms}

In this section, we give a brief overview of the theory of differential forms on the loop space $\L X$ of a Riemannian manifold $X$. 

\paragraph{The loop space.}
The {\em (smooth) loop space} of $X$ is the space $\L X = C^\infty(\T, X)$. It has the structure on an infinite-dimensional manifold, modelled on the Fréchet space $C^\infty(\T, \R^n)$. Here throughout, we denote $\T = S^1 = \R/ \Z$. For the tangent space at a loop $\gamma \in \L X$, we have the natural identifications
\begin{equation} \label{LoopTangentSpace}
  T_\gamma \L X \cong C^\infty(\T, \gamma^*TX),
\end{equation}
with the space of smooth vector fields along $\gamma$, i.e., sections of the pullback bundle $\gamma^*TX$ over $\T$. It is given by mapping a smooth variation $\gamma_s, s \in (-\varepsilon, \varepsilon)$, of loops to the vector field $t \mapsto \tfrac{\partial}{\partial s}|_{s=0} \gamma_s(t)$ along $\gamma = \gamma_0$. Dually, we have the identification
\begin{equation*}
  T^\prime_\gamma \L X = \mathscr{D}^\prime(\T, \gamma^*T^\prime X) \defeq \bigl( C^\infty(\T, \gamma^* TX) \bigr)^\prime
\end{equation*}
of the cotangent space with the space of $\gamma^*T^\prime X$-valued distributions on $\T$.

The tangent spaces carry a natural scalar product, given by the $L^2$ scalar product
\begin{equation} \label{L2ScalarProduct}
   \langle V, W\rangle_{L^2} \defeq \int_{S^1} \bigl\langle V(t), W(t) \bigr\rangle \dd t
\end{equation}
for $V, W \in T_\gamma \L X$. Of course, $T_\gamma \L X$ is not complete with respect to this scalar product; the completion with respect to the norm induced from \eqref{L2ScalarProduct} is the space $L^2(S^1, \gamma^* TX)$ of square-integrable vector fields along $\gamma$.

\paragraph{Differential forms on infinite-dimensional manifolds.} Let $Y$ be smooth manifold, modelled on a (possibly infinite-dimensional) locally convex space. It turns out that the ``correct'' definition of differential $\ell$-forms on $Y$ is  
\begin{equation} \label{DefinitionOfFormsInGeneral}
\Omega^\ell(Y) \defeq C^\infty\bigl(Y, L_{\mathrm{alt}}^\ell(TY, \R)\bigr),
\end{equation}
the space of smooth sections of the vector bundle  $L_{\mathrm{alt}}^\ell(TY, \R)$ over $Y$.  The fiber of this bundle at $y \in Y$ is the space of bounded, alternating multilinear functionals on $T_y Y$ \cite[\S33]{KrieglMichor}; here smoothness of a mapping always means that smooth curves are mapped to smooth curves (i.e. smoothness in the sense of convenient calculus \cite{KrieglMichor}). With this definition, one has a well-defined wedge product, exterior differential, pullback maps and Lie derivatives, just as in finite dimensions (in contrast to several other possible definitions, e.g., sections of the exterior power $\Lambda^\ell T^\prime Y$ of the cotangent bundle, see  \cite[33.21]{KrieglMichor}). 
Setting $Y = \L X$, we let
\begin{equation} \label{LoopSpaceForms}
  \Omega(\L X) = \bigoplus_{\ell=0}^\infty \Omega^\ell(\L X).
\end{equation}
be the algebra of differential forms on $\L X$.

\paragraph{The bundle $\boldsymbol{L_{\mathrm{alt}}^\ell(T\L X, \R)}$.} In particular, for the manifold $Y = \L X$, we  are lead to understand the spaces $L_{\mathrm{alt}}^\ell(T_\gamma \L X, \R)$ of bounded alternating $\ell$-linear functionals on $T_\gamma \L X = C^\infty(\T, \gamma^*TX)$, for each $\gamma \in \L X$; compare \eqref{LoopTangentSpace}. It turns out that these can be identified with a certain space of bundle-valued distributions on the torus $\T^\ell$. To this end, we need the $\ell$-th \emph{exterior} tensor product of the bundle $\gamma^*TX$ over $\T$ to obtain a bundle $\gamma^*TX^{\boxtimes \ell}$ over $\T^\ell$, the fiber of which at $\tau = (\tau_1, \dots, \tau_\ell) \in \T^\ell$ is
\begin{equation*}
(\gamma^*TX^{\boxtimes \ell})_\tau = T_{\gamma(\tau_1)}X \otimes \cdots \otimes T_{\gamma(\tau_\ell)} X.
\end{equation*}
The bundle $\gamma^*T^\prime X^{\boxtimes \ell}$ is constructed similarly. The relevant space of distributions is now the space
\begin{equation*}
  \mathscr{D}^\prime(\T^\ell, \gamma^*T^\prime X^{\boxtimes \ell}) \defeq \bigl(C^\infty(\T^\ell, \gamma^*TX^{\boxtimes \ell})\bigr)^\prime,
\end{equation*}
the dual space of the space of smooth sections of $\gamma^*TX^{\boxtimes \ell}$.
Any such distribution $\theta$ determines an element of $L^\ell(T_\gamma \L X, \R)$, again denoted by $\theta$, via the formula
\begin{equation} \label{DistributionInducesMultilinearFunctional}
  \theta(V_1, \dots, V_\ell) \defeq \langle \theta, V_1 \otimes \cdots \otimes V_\ell\rangle
\end{equation}
where the right hand side denotes the dual pairing of $\theta$ with the test section $V_1 \otimes \cdots \otimes V_\ell \in C^\infty(\T^\ell, \gamma^*TX^{\boxtimes \ell})$. 

\begin{lemma} \label{LemmaIdentificationOfFiberwiseForms}
    For each $\gamma \in \L X$, the identification \eqref{DistributionInducesMultilinearFunctional} provides an isomorphism of topological vector spaces
    \begin{equation*}
      L^\ell(T_\gamma \L X, \R) \cong \mathscr{D}^\prime(\T^\ell, \gamma^*TX^{\boxtimes \ell})
    \end{equation*}
\end{lemma}

\begin{proof}
It is easy to see that the assignment \eqref{DistributionInducesMultilinearFunctional} of an $\ell$-linear functional to a distribution is injective.

To see that each element of $L^\ell(T_\gamma \L X, \R)$ is given by a distribution, notice that any such element is by definition a bounded $\ell$-linear functional on $T_\gamma \L X = C^\infty(\T, \gamma^*TX)$ or, equivalently, a bounded {\em linear} form on the $\ell$-fold bornological tensor product $C^\infty(\T, \gamma^*TX)^{\otimes_\beta \ell}$. Since $C^\infty(\T, \gamma^*TX)$ is a nuclear Fr\'echet space, the bornological tensor product coincides with the projective tensor product \cite[Thm.~1.91]{Meyer} and bounded linear maps are continuous \cite[Thm.~1.29]{Meyer}. Hence we have
\begin{equation*}
  L^\ell(T_\gamma \L X, \R) = \Bigl(\underbrace{C^\infty(\T, \gamma^*T^\prime X) \otimes_\pi \cdots \otimes_\pi C^\infty(\T, \gamma^*T^\prime X)}_\ell\Bigr)^\prime,
\end{equation*}
where $\otimes_\pi$ denotes the projective tensor product. Using \cite[Thm.~44.1]{MR0225131}, we can identify the $\ell$-fold tensor product with $C^\infty(\T^\ell, \gamma^*TX^{\boxtimes \ell})$.
\end{proof}

By $\mathscr{D}^\prime(\T^\ell, \gamma^*T^\prime X^{\boxtimes \ell})^{S_\ell}$, we denote the subspace of distributions such that the $\ell$-linear functional defined by \eqref{DistributionInducesMultilinearFunctional} is alternating. To give an another description, notice that the space $C^\infty(\T^\ell, \gamma^*TX^{\boxtimes \ell})$ has a natural $S_\ell$-action by signed permutation of the factors, explicitly
\begin{equation} \label{SNActionOnSection}
  (\sigma \cdot V)(\tau_1, \dots, \tau_\ell) \defeq \sgn(\sigma) V(\tau_{\sigma_1}, \dots, \tau_{\sigma_\ell}).
\end{equation}
The space $\mathscr{D}^\prime(\T^\ell, \gamma^*T^\prime X^{\boxtimes \ell})^{S_\ell}$ can now equivalently be described as the dual space of the space of $S_\ell$-invariant sections $(C^\infty(\T^\ell, \gamma^*TX^{\boxtimes \ell})^{S_\ell}$. We conclude from Lemma~\ref{LemmaIdentificationOfFiberwiseForms} that for any $\gamma \in \L X$, we have the identification
\begin{equation*}
 L_{\mathrm{alt}}^\ell(T_\gamma \L X, \R) = \mathscr{D}^\prime(\T^\ell, \gamma^*T^\prime X^{\boxtimes \ell})^{S_\ell}.
\end{equation*}

\subsection{Preliminaries II: Spin Geometry} \label{SectionSpinPreliminaries}

In this section, we give a quick account of some notions of spin geometry needed in this paper.

\paragraph{Spin structures.} Let $Y$ be a manifold and let $\V$ be an $n$-dimensional oriented Euclidean vector bundle over $Y$. Equivalently, this means that the frame bundle of $\V$ has its structure group reduced to the special orthogonal group $\mathrm{SO}_n$. 

A {\em spin structure} for $\V$ is then a lift of the structure group to $\Spin_n \rightarrow \mathrm{SO}_n$, in other words, a $\Spin_n$-principal bundle $P \rightarrow \mathrm{SO}(\V)$ covering the special orthogonal frame bundle in a fashion compactible with the action of $\mathrm{SO}_n$  \cite[II\S1]{LawsonMichelsohn}. In the case that $Y = X$ is an oriented Riemannian manifold and $\V = TX$ is the tangent bundle of $X$, this gives the usual notion of a spin structure on the manifold $X$. 

We will also consider the case where $Y = \T$ and $\V = \gamma^*TX$, for $\gamma$ a loop in an $n$-dimensional oriented Riemannian manifold. In this case, spin structures are classified by $H^1(\T,  \Z_2) = \Z_2$, in particular, there are two isomorphism classes of spin structures on $\gamma^*TX$. Given a spin structure $P$, we obtain its {\em opposite} by setting $-P = P \times_{\Z_2} M$, where $M\rightarrow \T$ is the non-trivial $\Z_2$-principal bundle (the {\em Moebius bundle} or {\em Hopf bundle}) and we divide by the diagonal $\Z_2$-action.

\paragraph{The real spinor bundle.} Given a spin structure $P \rightarrow \mathrm{Fr}(\V)$ on an $n$-dimensional oriented Euclidean vector bundle $\V \rightarrow Y$, we can form the associated {\em (real) spinor bundle}
\begin{equation} \label{AssociatedSpinorBundle}
  \Sigma \defeq P \times_{\Spin_n} \Cl_n,
\end{equation}
where $\Cl_n = \Cl(\R^n)$ is the Clifford algebra on $\R^n$. Here (as usual) we identify $[p\cdot g, a] = [p, g \cdot a]$, $g \in \Spin_n$, inside $P \times \Cl_n$, where the action of $\Spin_n$ on $\Cl_n$ is by left multiplication, after realizing $\Spin_n$ inside the even part of the Clifford algebra.

The bundle $\Sigma$ defined in \eqref{AssociatedSpinorBundle} is naturally a bundle of graded $\Cl(\V)$-$\Cl_n$-bimodules on $Y$ (where the grading comes from the even/odd grading of the Clifford algebra). Here $\Cl(\V)$ is the bundle over $Y$ with fiber over $x$ the Clifford algebra on $\V_x$. We denote the corresponding {\em Clifford map} by
\begin{equation} \label{CliffordMap}
  \cc: \V \rightarrow \Cl(\V).
\end{equation}
It satisfies the usual Clifford relations
\begin{equation*}
  \cc(v)\cc(w) + \cc(w)\cc(v) = - 2 \langle v, w\rangle.
\end{equation*}
We denote by $\End_{\Cl_n}(\Sigma)$ the space of endomorphisms of $\Sigma$ that commute with the right action of $\Cl_n$. Such an endomorphism is always given by left multiplication by an element $a$ of the Clifford algebra $\Cl(\V)$ and we define its {\em supertrace} by the formula
\begin{equation} \label{DefinitionSupertrace}
  \str(a) = 2^{n/2} \langle a, \cc(e_1) \cdots \cc(e_n)\rangle,
\end{equation}
for  an oriented orthonormal basis $e_1, \dots, e_n$ of $\V$. This is indeed a supertrace, in the sense that it has the graded cyclic permutation property
\begin{equation} \label{CyclicPermutation}
   \str(ab) = (-1)^{|a||b|}\str(ba)
\end{equation}
on homogeneous elements $a, b \in \Cl(\V)$.

\paragraph{Comparison to the complex spinor bundle.} The advantage of the real spinor bundle is that it is graded in any dimension. If $P\rightarrow Y$ is a $\Spin_n$-principal bundle, we can also form the {\em complex} spinor bundle by
\begin{equation*}
  \Sigma_\C \defeq P \times_{\Spin_n} S,
\end{equation*}
where $S$ is the complex spinor representation. In the case that $n$ is even, $S$ is graded, which induces a grading $\Sigma_\C = \Sigma_\C^+ \oplus \Sigma_\C^-$ on the complex spinor bundle. We can therefore take the operator supertrace $\str_\C$, which is related to  \eqref{DefinitionSupertrace} by the formula
\begin{equation} \label{TraceComparisonEven}
  \str_\C(a) = (-i)^{n/2}\str(a)
\end{equation}
for $a \in \End(\Sigma_\C) \cong \Cl_n \otimes \C$ (here $\str$ as defined in \eqref{DefinitionSupertrace} is complex linearly extended to $\Cl_n \otimes \C$). In contrast, if $n=2m+1$ is odd, then we have the formula
\begin{equation} \label{TraceComparisonOdd}
  \tr_\C(a) = i(2i)^m \langle a, \cc(e_1) \cdots \cc(e_n)\rangle + 2^m \langle a, \mathbf{1}\rangle
\end{equation}
for the endomorphism trace of $a \in \Cl(\V) \subset \End(\Sigma_\C)$.

\section{Atiyah's formula} \label{SectionAtiyahsFormula}

Let $M$ be an oriented Riemannian manifold. In this section, we calculate the parallel transport in the spinor bundle and the zeta-regularized determinant of the covariant derivative along a loop, both in terms of the parallel transport in the tangent bundle. These calculations are the starting point for our construction of the top degree map. In particular, we establish the formula 
\begin{equation} \label{AtiyahsFormula}
  \det\nolimits_\zeta(\nabla_{\dot{\gamma}}) = \str\bigl([\gamma\|_1^0]^\Sigma)^{2},
\end{equation}
which,  as far was we know, was first given by Atiyah \cite{Atiyah}.

\paragraph{The parallel transport in the spinor bundle.} Let  $P$ be a spin structure on $\gamma^*TX \rightarrow \T$ with associated  real spinor bundle $\Sigma = P \times_{\Spin_n} \Cl_n$, which always exists since $M$ and hence $\gamma^*TX$ is oriented. Since the covering $\Spin_n \rightarrow \mathrm{SO}_n$ is discrete, the Levi-Civita connection on $\gamma^*TX$ (obtained by pullback from the Levi-Civita connection on $TX$) lifts in a unique way to a connection on $\Sigma$. Let $[\gamma\|_1^0]^\Sigma \in \End_{\Cl_n}(\Sigma_{\gamma(0)})$ by the corresponding parallel transport around $\T$. It will be important for our considerations to calculate the parallel transport in the spinor bundle in terms of the parallel transport $[\gamma\|_1^0]^{TX}$ in the tangent bundle. Since $[\gamma\|_1^0]^{TX}$ is an orthogonal endomorphism of $T_{\gamma(0)} X$, there is an orthonormal basis $e_1, \dots, e_n$ of $T_{\gamma(0)}X$ such that the parallel transport is given by the matrix
\begin{equation} \label{MatrixOfParallelTransport}
  [\gamma\|_1^0]^{TX} ~\widehat{=} ~
  \begin{pmatrix}
    \cos(2\pi \alpha_1) & -\sin(2\pi \alpha_1) &  & & & & & \\
    \sin(2\pi \alpha_1) & \cos(2\pi \alpha_1) &  & & & & & \\
    & & \ddots & & & & & \\
    & & & \cos(2\pi \alpha_{m}) & -\sin(2\pi \alpha_m) & &  & \\
    & & & \sin(2\pi \alpha_{m}) & \cos(2\pi \alpha_m) & &  & \\
    & & &  &  & 1 & &  \\    
    & & &  &  & & \ddots &  \\    
    & & & &   & & & 1  \\  
    \end{pmatrix}
\end{equation}
with respect to this basis, where $ \alpha_1, \dots, \alpha_m \notin \Z$ are real numbers. Using the formula $[\gamma\|_s^t]^{T^\prime X})^\prime = [\gamma\|_t^s]^{TX}$ for the parallel transport in the cotangent bundle, it is not hard to see that the parallel transport $[\gamma\|_1^0]^{T^\prime X}$ is given by the \emph{same} matrix \eqref{MatrixOfParallelTransport} with respect to the dual basis $e_1^\prime, \dots, e_n^\prime$. 
We now have the following lemma.

\begin{lemma} \label{LemmaSpinParallelTransport}
Let $e_1, \dots, e_n$ be an orthonormal basis of $T_{\gamma(0)} X$ with respect to which $[\gamma\|_1^0]^{TX}$ takes the form \eqref{MatrixOfParallelTransport} for numbers $\alpha_1, \dots, \alpha_m  \notin \Z$, and let $\cc_1, \dots, \cc_n$ be the corresponding generators of the Clifford algebra $\Cl(T_{\gamma(0)}X)$. Then the parallel transport around $\gamma$ in the spinor bundle $\Sigma$ is given by
  \begin{equation} \label{Eq:FormulaSpinorParallelTransport}
     [\gamma\|_1^0]^{\Sigma} = \epsilon_0 \prod_{j=1}^m \bigl(\cos(\pi \alpha_j) + \sin(\pi \alpha_j) \cc_{2j-1} \cc_{2j} \bigr),
  \end{equation}
  for some $\epsilon_0 \in \{\pm 1\}$. Moreover, when the spin structure $P$ is is replaced to $-P$, then $\epsilon_0$ changes its sign.
\end{lemma}

\begin{proof}
Remember that $\Sigma = P \times_{\Spin_n} \Cl_n$. Hence with respect to a frame $p \in P_{\gamma(0)}$, the parallel transport is given in terms of an element $g \in \Spin_n$, such that
\begin{equation*}
  [\gamma\|_1^0]^\Sigma [p, \psi] = [p, g \cdot \psi] \qquad \text{for all}~~\psi \in \Cl_n.
\end{equation*}
If $\rho_P: P \rightarrow \mathrm{Fr}(\gamma^*TX)$ is the projection, then the parallel transport in the tangent bundle  is given by an analogous formula
\begin{equation*}
  [\gamma\|_1^0]^{TX} [\rho_P(p), v] = [\rho_P(p), Q v],
\end{equation*}
for some $Q \in \mathrm{SO}_n$. In fact, due to the compatibility of the connections, we have $Q = \rho(g)$, where $\rho: \Spin_n \rightarrow \mathrm{SO}_n$ is the standard covering. 

We therefore need to show that if $\rho_P(p)$ is such that $Q$ has the form \eqref{MatrixOfParallelTransport}, then any  preimage under $\rho$ in $\Spin_n$ is given by the formula \eqref{Eq:FormulaSpinorParallelTransport}. This follows from verifying that if one conjugates a vector $\cc(v) = v_1 \cc_{2j-1} + v_2 \cc_{2j}$ in the Clifford algebra $\Cl(T_{\gamma(0)}X)$ by the element $\pm(\cos(\pi \alpha_j) + \sin(\pi \alpha_j) \cc_{2j-1}\cc_{2j})$ to obtain $w_1 \cc_{2j-1} + w_2 \cc_{2j} = \cc(w) \in \Cl(T_{\gamma(0)}X)$, then $w$ is given by
\begin{equation*}
  \begin{pmatrix}w_1 \\ w_2 \end{pmatrix} ~\widehat{=}~ \begin{pmatrix} \cos(2\pi \alpha_j) & - \sin(2 \pi \alpha_j) \\ \sin(2 \pi \alpha_j) & \cos(2 \pi \alpha_j) \end{pmatrix} 
  \begin{pmatrix}v_1 \\ v_2 \end{pmatrix}.
\end{equation*}
We have $-P = (P \times M)/\Z_2$, where $M$ is the M\"obius bundle, and we divide out the diagonal action of $\Z_2$. Since parallel transport around $\T$ in $M$ adds a sign, this adds a sign in formula \eqref{Eq:FormulaSpinorParallelTransport}.
\end{proof}

Since right multiplication by $\Cl_n$ is parallel, it commutes with $[\gamma\|_1^0]^{\Sigma}$; this also follows from the explicit formula \eqref{Eq:FormulaSpinorParallelTransport}.  Combining \eqref{Eq:FormulaSpinorParallelTransport} with \eqref{DefinitionSupertrace}, we obtain the following immediate consequence.

\begin{corollary} \label{CorollarySupertrace}
If $m = n/2$, then
\begin{equation} \label{FormulaSupertrace}
  \str\bigl([\gamma\|_1^0]^{\Sigma}\bigr) =  \epsilon_0 \cdot \mathrm{sign}(e_1, \dots, e_n) \cdot  \prod_{j=1}^{n/2} 2\sin(\pi \alpha_j),
\end{equation}
where $\mathrm{sign}(e_1, \dots, e_n) = \pm 1$, depending on whether the basis $e_1, \dots, e_n$ with respect to which $[\gamma\|_1^0]^{TX}$ is given by \eqref{MatrixOfParallelTransport} is positively oriented or not. If $m < n/2$, in particular if $n$ is odd, then the supertrace vanishes.
\end{corollary}

\paragraph{The determinant of the covariant derivative.} As mentioned above, for $\gamma \in \L X$, the Riemannian structure of $X$ induces a connection on the pullback $\gamma^* TX$ of the tangent bundle over $\T$. We denote by $\nabla_{\dot{\gamma}}$ the operator acting on sections of $\gamma^*T X$ by differentiating in direction of the canonical vector field $\partial_t$ on $\T$ using this pullback connection.
Integrating by parts, we obtain that
\begin{equation} \label{SkewSymmetriyOmega}
  \langle\nabla_{\dot{\gamma}} V, W\rangle_{L^2} = \int_{\T} \langle \nabla_{\dot{\gamma}} V(t), W(t)\rangle \dd t = - \int_{\T} \langle  V(t), \nabla_{\dot{\gamma}}W(t)\rangle \dd t = - \langle V, \nabla_{\dot{\gamma}} W\rangle_{L^2}
\end{equation}
for vector fields $V, W \in T_\gamma \L X$  $= C^\infty(\T, \gamma^\ast TX)$, hence $\nabla_{\dot{\gamma}}$ is skew-symmetric on this domain. In fact, considered as an unbounded operator on $L^2(\T, \gamma^* TX)$ with domain $C^\infty(\T, \gamma^* TX)$, it is {\em essentially skew-adjoint}, meaning that it has a unique closed skew-adjoint extension; the domain of this extension is the Sobolev space $H^1(\T, \gamma^*TX)$, the space of absolutely continuous vector fields along $\gamma$ with derivative contained in $L^2$. 

The {\em zeta function} of $\nabla_{\dot{\gamma}}$ is the function 
\begin{equation} \label{ZetaOfNabla}
  \zeta_{\gamma}(s) \defeq \sum_{\lambda \neq 0} \lambda^{-s} = 2\sum_{\lambda \neq 0} \cos\left(\frac{\pi s}{2}\right) |\lambda|^{-s},
\end{equation}
where the sum goes over all non-zero eigenvalues of $\nabla_{\dot{\gamma}}$ (as an operator on the complexification $L^2(\T, \gamma^*TX)\otimes \C$). To define powers of complex numbers, we use a spectral cut at the negative real axis; the second equality in \eqref{ZetaOfNabla} then uses that the non-zero eigenvalues of $\nabla_{\dot{\gamma}}$ are purely imaginary and come in complex conjugate pairs since $\nabla_{\dot{\gamma}}$ is real, i.e., commutes with complex conjugation. The sum defines a holomorphic function for $\mathrm{Re}(s)>1$. This function has a meromorphic extension to all of $\C$, which is regular at zero. The (reduced) zeta determinant of $\nabla_{\dot{\gamma}}$ is then defined by
\begin{equation*}
  \det\nolimits_\zeta^\prime(\nabla_{\dot{\gamma}}) = e^{-\zeta_\gamma^\prime(0)},
\end{equation*}
motivated by the fact that the right hand side is formally the product of the non-zero eigenvalues, if one pretends for a moment that there are only finitely many eigenvalues. The zeta determinant itself is defined by
\begin{equation*}
\det\nolimits_\zeta(\nabla_{\dot{\gamma}}) = \begin{cases} \det\nolimits_\zeta^\prime(\nabla_{\dot{\gamma}}) & 0 \notin \text{spectrum of }~\nabla_{\dot{\gamma}}\\ 0 & 0 \in \text{spectrum of }~\nabla_{\dot{\gamma}}.\end{cases}
\end{equation*}
We now have the following proposition, which can be found e.g., in \cite[Lemma~2]{Atiyah}, \cite[Eq.~(2.13)]{Bismut1} and \cite[pp.~125ff]{PratWaldron}. Together with Corollary~\ref{CorollarySupertrace}, this yields Atiyah's formula \eqref{AtiyahsFormula}.

\begin{proposition} \label{Prop:ZetaPfaffian}
Let $e_1, \dots, e_n$ be an orthonormal basis of $T_{\gamma(0)}X$ with respect to which $[\gamma\|_1^0]^{TX}$ is given by the formula \eqref{MatrixOfParallelTransport}. Then the reduced zeta determinant of $\nabla_{\dot{\gamma}}$ is given by the formula
  \begin{equation} \label{Eq:ExplicitZetaPfaffian}
    \det\nolimits_\zeta^\prime(\nabla_{\dot{\gamma}}) = \prod_{j=1}^m 4 \sin(\pi \alpha_j)^2.
  \end{equation}
\end{proposition}

\begin{remark}
Note that the unreduced zeta determinant $\det\nolimits_\zeta(\nabla_{\dot{\gamma}})$ is always zero if $n$ is odd, because in odd dimensions, there is always a parallel vector field around any loop.
\end{remark}

\section{The Pfaffian and the Spin Line Bundle} \label{SectionLineBundles}

In this section, we introduce two real line bundles on the loop space of an oriented Riemannian manifold $X$: The Pfaffian line bundle $\Pf$ associated to $\nabla_{\dot{\gamma}}$, and the line bundle $\mathcal{L} = \hat{\L} X \times_{\Z_2} \R$ \NZ{use widehat ?} associated to the $\Z_2$-principal bundle which is the transgression of the spin lifting gerbe on $X$ (defined in \cite{Murrey} or \cite{Waldorf1}; see also \cite{StolzTeichner}).  We then construct a canonical isomorphism between these two bundles which also preserves the natural metrics on the two.
This was previously done by Prat-Waldron in his thesis \cite{PratWaldron}. Since we will use a slightly different description of $\mathcal{L}$, our definition of $\Phi$ differs slightly from his.

{\paragraph{The Pfaffian line bundle.} We start by defining the Pfaffian line bundle\footnote{For further details, see e.g., \cite[\S4.3]{PratWaldron}}. To this end, for $a>0$, define the open sets
\begin{equation*}
 U^{(a)} \defeq \{\gamma \in \L X \mid 4 \pi^2 a^2 \notin \mathrm{spec}(-\nabla_{\dot{\gamma}}^2)\}
\end{equation*}
of $\L X$. These form an open cover of $\L X$. On the open sets $U^{(a)}$ respectively the intersections $U^{(a)} \cap U^{(b)}$, we define
\begin{equation} \label{DefEBundles}
  \mathcal{E}^{(a)} \defeq \bigoplus_{\lambda < a} \mathrm{Eig}(-\nabla_{\dot{\gamma}}^2, 4 \pi^2 \lambda^2)^\prime \quad \text{respectively} \quad \mathcal{E}^{(a, b)} \defeq \bigoplus_{a <\lambda < b} \mathrm{Eig}(-\nabla_{\dot{\gamma}}^2, 4 \pi^2 \lambda^2)^\prime,
\end{equation}
the direct sum of the duals of the eigenspaces of the Laplacian $-\nabla_{\dot{\gamma}}^2$ corresponding to the respective intervals.
These are smooth vector bundles of finite rank on $U^{(a)}$ respectively $U^{(a)}\cap U^{(b)}$. Note that the operators $(\nabla^\prime_{\dot{\gamma}})^2$ and $\nabla_{\dot{\gamma}}^2$ have the same eigenvalues, where $\nabla^\prime_{\dot{\gamma}}$ denotes the covariant derivative on the dual bundle $\gamma^*T^\prime X$. Using the metric, we therefore can therefore identify 
$\mathrm{Eig}(-\nabla_{\dot{\gamma}}^2, 4 \pi^2 \lambda^2)^\prime \cong \mathrm{Eig}(-(\nabla_{\dot{\gamma}}^\prime)^2, 4 \pi^2 \lambda^2)$, so that we get canonical isomorphisms
\begin{equation*}
\mathcal{E}^{(a)} \oplus \mathcal{E}^{(a, b)} = \mathcal{E}^{(b)}.
\end{equation*}
 The associated determinant line bundles $\Lambda^{\mathrm{top}}\mathcal{E}^{(a)}$ are then real line bundles on the open sets $U^{(a)}$ and satisfy $\Lambda^{\mathrm{top}}\mathcal{E}^{(a)} \otimes \Lambda^{\mathrm{top}}\mathcal{E}^{(a, b)} = \Lambda^{\mathrm{top}}\mathcal{E}^{(b)}$.

We now discuss how these line bundles glue together to a line bundle $\Pf$ on $\L X$. Because $\nabla_{\dot{\gamma}}$ is an invertible, skew-adjoint operator $\mathcal{E}^{(a, b)}$ for all $0 < a < b$, each of these spaces has even dimension. Moreover, the Pfaffian $\pf(\nabla_{\dot{\gamma}}^{(a, b)})$ of $\nabla_{\dot{\gamma}}$  restricted to $\mathcal{E}^{(a, b)}$ (as defined in \S\ref{SectionTopDegree}) is a well-defined non-zero element of $\Lambda^{\top}\mathcal{E}^{(a, b)}$ and these elements satisfy
\begin{equation} \label{CocycleCondition}
  \pf(\nabla_{\dot{\gamma}}^{(a, b)}) \otimes \pf(\nabla_{\dot{\gamma}}^{(b, c)}) = \pf(\nabla_{\dot{\gamma}}^{(a, c)})
\end{equation}
for $a < b < c$. We remark that $\pf(\nabla_{\dot{\gamma}}^{(a, b)})$ is the $\frac{\dim(\mathcal{E}^{(a, b)})}{2}$-th power of the canonical two form $\omega$ (defined in \eqref{CanonicalTwoForm}) restricted to $\mathcal{E}^{(a, b)}$. In particular, we obtain isomorphisms
\begin{equation*}
\begin{aligned}
  h^{(a, b)}: \Lambda^{\top} \mathcal{E}^{(a)} &\longrightarrow \Lambda^{\top} \mathcal{E}^{(a)} \otimes \Lambda^{\top} \mathcal{E}^{(a, b)} = \Lambda^{\top} \mathcal{E}^{(b)}, ~~~~~~~~\beta \longmapsto \beta \otimes \pf(\nabla_{\dot{\gamma}}^{(a, b)})
\end{aligned}
\end{equation*}
on the overlaps $U^{(a)} \cap U^{(b)}$
By \eqref{CocycleCondition}, these satisfy the cocycle condition $h^{(b, c)} \circ h^{(a, b)} = h^{(a, c)}$ on triple intersections $U^{(a)} \cap U^{(b)} \cap U^{(c)}$ so that the line bundles $\Lambda^{\top} \mathcal{E}^{(a)}$ together with the glueing isomorphisms $h^{(a, b)}$ indeed define a real line bundle $\Pf$ over $\L X$. 

The line bundle $\Pf$ has a natural metric. Notice that the vector bundles $\mathcal{E}^{(a)}$ and $\mathcal{E}^{(a, b)}$ are naturally equipped with the $L^2$ metric, which also induces a metric on $\Lambda^{\top}\mathcal{E}^{(a)}$. However, these metrics are not compatible with the glueing isomorphisms $h^{(a, b)}$. Instead, over $U^{(a)}$, we consider the metric on $\Lambda^{\top}\mathcal{E}^{(a)}$ given by\footnote{Here we assume $\dim \mathcal{E}^{(a)} = N$.}
\begin{equation} \label{MetricOnPf}
   \bigl\langle \theta_1 \wedge \dots \wedge \theta_N, \tilde{\theta}_1 \wedge \dots \wedge \tilde{\theta}_N\bigr\rangle_{\Pf} \defeq \det \Bigl( \langle\theta_a, \tilde{\theta}_b\rangle_{L^2}\Bigr)_{1 \leq a, b \leq N} \det\nolimits_\zeta\bigl(\nabla_{\dot{\gamma}}^{(a, \infty)}\bigr),
\end{equation}
for one forms $\theta_a, \tilde{\theta}_a \in \mathcal{E}^{(a)} \subset L^2(\T, \gamma^*T^\prime X)$. This is the metric on $\Lambda^{\top}\mathcal{E}^{(a)}$ induced by the $L^2$ metric on $\mathcal{E}^{(a)}$, but modified by the factor
\begin{equation*}
\det\nolimits_\zeta\bigl(\nabla_{\dot{\gamma}}^{(a, \infty)}\bigr) \defeq \frac{\det\nolimits_\zeta^\prime\bigl(\nabla_{\dot{\gamma}}\bigr)}{\det(\nabla_{\dot{\gamma}}^{(0, a)})}.
\end{equation*}
That this metric is well defined as a metric on $\Pf$, i.e., compatible with the glueing isomorphisms $h^{(a, b)}$, follows from the facts that we have $\|\pf(\nabla_{\dot{\gamma}}^{(a, b)})\|_{L^2}^2 = \det(\nabla_{\dot{\gamma}}^{(a, b)})$ and $\det(\nabla_{\dot{\gamma}}^{(0, a)})\det(\nabla_{\dot{\gamma}}^{(a, b)}) = \det(\nabla_{\dot{\gamma}}^{(0, b)})$. Since $\Pf$ is a real line bundle, there is a unique connection  compatible with this metric, so we have completed the description of $\Pf$ as a geometric real line bundle.

There is a canonical section $\pf(\nabla_{\dot{\gamma}})$,  of this bundle. Over $U^{(a)}$, it is given by the section 
\begin{equation} \label{DefinitionPfaffian}
\pf(\nabla_{\dot{\gamma}})^{(a)} \defeq \pf(\nabla_{\dot{\gamma}}^{(a)})
\end{equation} 
of $\Lambda^{\top}\mathcal{E}^{(a)}$ where $\nabla_{\dot{\gamma}}^{(a)}$ is the restriction of $\nabla_{\dot{\gamma}}$ to $\mathcal{E}^{(a)}$. By the multiplication rule for the Pfaffian, we have $h^{(a, b)} \pf(\nabla_{\dot{\gamma}})^{(a)} = \pf(\nabla_{\dot{\gamma}})^{(b)}$ so that the $\pf(\nabla_{\dot{\gamma}})^{(a)}$ glue together to a section of $\Pf$. This section $\pf(\nabla_{\dot{\gamma}})$ satisfies
\begin{equation} \label{NormOfPfaffianSection}
  \|\pf(\nabla_{\dot{\gamma}})\|_{\Pf}^2 = \det\nolimits_\zeta(\nabla_{\dot{\gamma}}).
\end{equation}

Given $\gamma \in \L X$, we can also define the {\em reduced Pfaffian} $\pf^\prime(\nabla_{\dot{\gamma}})$ by requiring that over $U^{(a)}$, it is given by
\begin{equation} \label{DefReducedPfaffian}
\pf^\prime(\nabla_{\dot{\gamma}})^{(a)} \defeq \pf(\nabla_{\dot{\gamma}}^{(0, a)})  \in \Lambda^{\mathrm{top}} \mathcal{E}^{(0, a)}.
\end{equation}
for $a$ such that $4 \pi^2 a^2 \notin \mathrm{spec}(-\nabla_{\dot{\gamma}}^2)$. By varying $a$, this glues together to an element of the line $\Pf_\gamma \otimes (\Lambda^{\mathrm{top}}\ker(\nabla_{\dot{\gamma}}))$.

{\paragraph{The spin line bundle.}} The bundle $\hat{\L} X$ is easier to describe: Over a path $\gamma$, the fiber is the set of isomorphism classes of spin structures on $\gamma^* TX$, i.e., $\Spin_n$-principal bundles $P \rightarrow \T$ that are compatible with the frame bundle of $\gamma^* TX$ (see \cite[Ch.~2, \S 1]{LawsonMichelsohn}). Since $X$ is oriented, $\gamma^* TX$ is always trivial, so by Corollary~1.5 in \cite{LawsonMichelsohn}, there are exactly two isomorphism classes of spin structures for each path $\gamma$. Letting $\Z_2$ act by fiberwise exchanging the spin structures, this turns $\hat{\L} X$ into a $\Z_2$ principal bundle. Clearly, a spin structure $P^{\Spin}$ on $X$ defines a section $s$ of $\hat{\L} X$, by defining $s(\gamma) = \gamma^*P^{\Spin}$, and $\hat{\L} X$ is trivial if and only if $X$ is spin.  

The smooth structure of $\hat{\L}X$ is defined as follows. Given a loop $\gamma \in \L X$, we can always find a tubular neighborhood $U$ of its image, i.e., a neighborhood diffeomorphic to $\T \times \R^{n-1}$ (this uses that $X$ is orientable). The set $\mathcal{U}$ of loops with image contained in $U$ is then an open neighborhood of $\gamma$ in $\L X$. On the other hand, there exists a spin structure $P$ on $TX|_U \cong T(\T \times \R^{n-1})$, and the smooth structure is fixed by the requirement that the section $s$ over $\mathcal{U}$ with $s(\gamma) = \gamma^* P$ be smooth.

\begin{remark}
A word of warning: While a spin structure gives a trivialization of $\hat{\L} X$, it is {\em not} true that any non-vanishing section of $\hat{\L} X$ defines a spin structure on $X$. Instead, in order to obtain a spin structure this way, the section has to satisfy the additional condition of being compatible with the {\em fusion product} on $\L X$ (compare \cite{StolzTeichner}).
\end{remark}

We can now define the line bundle $\mathcal{L} = \hat{\L} X \times_{\Z_2} \R$ by performing the associated bundle construction via the non-trivial action of $\Z_2$ on $\R$. Hence elements of this bundle are equivalence classes $[P, \lambda]$, where $\lambda \in \R$ and $P$ is (an equivalence class of) a spin structure on $\gamma^* TX$; the equivalence relation then identifies $(P, \lambda) \sim (-P, - \lambda)$, where $-P$ denotes the opposite spin structure. A metric on $\mathcal{L}$ is defined by setting
\begin{equation*}
 \bigl\|[P, \lambda]\bigr\|_{\mathcal{L}} \defeq |\lambda|
\end{equation*}
for $P \in \hat{\L} X$ and $\lambda \in \R$.

There is a canonical section $q_0$ of $\mathcal{L}$, constructed as follows. After choosing a spin structure $P$ for $\gamma^*TX$, we can form the vector bundle $\Sigma = \Sigma_P$ as in \eqref{AssociatedSpinorBundle}. As discussed in \S\ref{SectionAtiyahsFormula}, one obtains a connection on $\Sigma$, and a corresponding parallel transport $[\gamma\|_0^q]^{\Sigma_P}$ around $\gamma$. The desired section is now given by 
\begin{equation*}
q_0(\gamma) \defeq \str\bigl( [\gamma\|_1^0]^\Sigma\bigr) \defeq \bigl[P, \str \bigl( [\gamma\|_1^0]^{\Sigma_P}\bigr)\bigr].
\end{equation*}
This does not depend on the choice of $P$: If one changes $P$ to $-P$, then by Lemma~\ref{LemmaSpinParallelTransport}, $\str([\gamma\|_1^0]^\Sigma)$ changes its sign, so $q_0$ is a well-defined section $\mathcal{L}$. It is smooth since $[\gamma\|_1^0]^\Sigma$ depends smoothly on $\gamma$, having fixed a spin structure in a tubular neighborhood of the image of $\gamma$ in $X$ (compare the definition of the smooth structure above).
Notice that $q_0$ vanishes if $n$ is odd, since then the super trace is an odd functional, while $[\gamma\|_1^0]^\Sigma$ is always an even endomorphism.

{\paragraph{The canonical isomorphism.}} We will now show that the bundles $\Pf$ and $\mathcal{L}$ are isomorphic as geometric line bundles, by constructing an explicit isomorphism $\Phi: \mathcal{L} \rightarrow \Pf$, defined in \eqref{DefinitionPhiGamma} below. 

We start with some preparations. For a fixed loop $\gamma$, let $e_1, \dots, e_n$ be an orthonormal basis of $T_{\gamma(0)}X$ such that that the parallel transport $[\gamma\|_1^0]^{TX}$ takes the form \eqref{MatrixOfParallelTransport} with respect to this basis, for numbers $\alpha_1, \dots, \alpha_m \notin \Z$. Let $e_1^\prime, \dots, e_n^\prime \in T_{\gamma(0)}^\prime X$ be the dual basis and let 
\begin{equation} \label{ParallelTranslatesEj324}
E_j(t) = [\gamma\|_1^t]^{TX}e_j, \qquad \text{respectively} \qquad E_j^\prime(t) = [\gamma\|_1^t]^{T^\prime X} e_j^\prime
\end{equation}
be the corresponding parallel translates. The one forms $E^\prime_{2m+1}, \dots, E_n^\prime$ then form an orthonormal basis of $\ker(\nabla_{\dot{\gamma}}^\prime)$.
For $a >0$ such that $\gamma \in U^{(a)}$, define $\Theta_\gamma^{(a)} \in \Lambda^{\mathrm{top}}\mathcal{E}_\gamma^{(a)} = \Lambda^{\mathrm{top}}\ker(\nabla_{\dot{\gamma}}) \otimes \Lambda^{\mathrm{top}} \mathcal{E}^{(0, a)}$ by
\begin{equation} \label{DefinitionTheta}
\Theta_\gamma^{(a)} ~\defeq~ \prod_{j=1}^m \frac{-1}{2  \sin(\pi\alpha_j)} E^\prime_{2m+1} \wedge \dots \wedge E^\prime_{n} \wedge \pf(\nabla_{\dot{\gamma}}^{(0, a)}).
\end{equation}
For $b>a$ such that $\gamma \in U^{(b)}$, we have $h^{(a, b)}\Theta_\gamma^{(a)} = \Theta_\gamma^{(b)}$, so the $\Theta_\gamma^{(a)}$ glue together to an element $\Theta_\gamma \in \Pf_\gamma$. This element depends on the choice of basis $e_1, \dots, e_n$ and the corresponding choice of numbers $\alpha_1, \dots, \alpha_m \notin \Z$. By \eqref{NormOfPfaffianSection} and Prop.~\ref{Prop:ZetaPfaffian}, we always have 
\begin{equation} \label{NormOfTheta}
\|\Theta_\gamma\|_{\Pf}^2 = \prod_{j=1}^m \frac{1}{4 \sin(\pi\alpha_j)^2} \cdot \underbrace{\|\pf(\nabla_{\dot{\gamma}}\|_{\Pf}^2}_{=\det^\prime_\zeta(\nabla_{\dot{\gamma}})} = 1,
\end{equation}
so that the elements $\Theta_\gamma$ obtained this way differ at most by a sign.

\begin{definition}[The isomorphism]
Now fixing an orthonormal basis $e_1, \dots, e_n$ and numbers $\alpha_1, \dots \alpha_m \notin \Z$ as before, we define an isometric isomorphism $\Phi_\gamma: \mathcal{L}_\gamma \longrightarrow \Pf_\gamma$ as follows. Given a spin structure $P$ on $\gamma^* TX$ with associated spinor bundle $\Sigma_P$, let $\epsilon_0$ be determined by formula \eqref{Eq:FormulaSpinorParallelTransport}. Then define
\begin{equation} \label{DefinitionPhiGamma}
  \Phi_\gamma([P, \lambda]) \defeq  \lambda \cdot \epsilon_0 \cdot \mathrm{sign}(e_1, \dots, e_n) \cdot \Theta_\gamma,
\end{equation}
where $\mathrm{sign}(e_1, \dots, e_n) = \pm 1$, depending on whether the basis is positively or negatively oriented and $\Theta \in \Pf_\gamma$ is defined by \eqref{DefinitionTheta}. The definition of $\Phi_\gamma$ is independent from the choice of the representatives of $[P, \lambda]$, because $\epsilon_0$ changes to $- \epsilon_0$ if $P$ changes to $-P$. 
\end{definition}

\begin{theorem}
  For each $\gamma$, the isomorphism $\Phi_\gamma$ defined above is independent of the choice of $e_1, \dots, e_n$ and $\alpha_1, \dots, \alpha_m  \notin \Z$. Moreover, the $\Phi_\gamma$ assemble to a smooth isometric isomorphism of line bundles $\Phi: \mathcal{L} \rightarrow \Pf$.
\end{theorem}

\begin{proof}
Let $\tilde{e}_1, \dots, \tilde{e}_n$ be a different orthonormal basis of $T_{\gamma(0)} X$, with respect to which $[\gamma\|_1^0]^{TX}$ takes the form \eqref{MatrixOfParallelTransport}, with possibly different numbers $\tilde{\alpha}_1, \dots, \tilde{\alpha}_m  \notin \Z$. Denote the corresponding quantities appearing in formula \eqref{DefinitionPhiGamma} by $\tilde{\epsilon}_0$ respectively $\tilde{\Theta}_\gamma$. 

We first investigate the following special cases.
\begin{enumerate}[(1)]
\item Suppose first we have only changed the $\alpha_j$, i.e., that $\tilde{e}_j = e_j$ for all $j$ and that $\tilde{\alpha}_j = \alpha_{j}+k_j$, for numbers $k_j \in \Z$. Then we have $\cos(\pi \tilde{\alpha}_j) = (-1)^{k_j} \cos(\pi \alpha_j)$ and $\sin(\pi \tilde{\alpha}_j) = (-1)^{k_j} \sin(\pi \alpha_j)$, hence
\begin{equation*}
\tilde{\epsilon}_0 = (-1)^{k_1 + \dots + k_m}\epsilon_0, ~~~~\text{but also}~~~~\tilde{\Theta}_\gamma = (-1)^{k_1 + \dots + k_m}\Theta_\gamma.
\end{equation*}
Hence $\Phi_\gamma$ is invariant with respect to this change. 

\item Now fix $j_0 \in \{1, \dots, m\}$ and suppose that $\tilde{e}_j = e_j$ for $j \notin\{ 2j_0-1, 2j_0\}$ and $\tilde{\alpha}_{j} = \alpha_{j}$ for $j \neq j_0$, while $\tilde{e}_{2j_0-1} = e_{2j_0}$, $\tilde{e}_{2j_0} = e_{2j_0-1}$ and $\tilde{\alpha}_{j_0} = -\alpha_{j_0}$. Then clearly 
\begin{equation*}
\mathrm{sign}(\tilde{e}_1, \dots, \tilde{e}_n) = -\mathrm{sign}(e_1, \dots,  e_n)
\end{equation*}
 and since $\sin(\pi \tilde{\alpha}_{j_0}) = -\sin(\pi \alpha_{j_0})$, this is compensated by $\tilde{\Theta}_\gamma = - \Theta_\gamma$. Furthermore, since by the Clifford multiplication rules $\cc_{2j_0} \cc_{2j_0-1} = - \cc_{2j_0-1} \cc_{2j_0}$, we have
\begin{equation*}
\begin{aligned}
\cos(\pi \tilde{\alpha}_{j_0}) + \sin(\pi \tilde{\alpha}_{j_0}) \tilde{\cc}_{2j_0-1} \tilde{\cc}_{2j_0}  &= \cos(\pi \alpha_{j_0}) - \sin(\pi \alpha_{j_0}) \cc_{2j_0} \cc_{2j_0-1}\\
&= \cos(\pi \alpha_{j_0}) + \sin(\pi \alpha_{j_0}) \cc_{2j_0-1} \cc_{2j_0},
\end{aligned}
\end{equation*}
hence $\tilde{\epsilon}_0 = \epsilon_0$. 
\item If there exists a permutation $\sigma \in S_m$ such that $\tilde{e}_{2j-1}=e_{2\sigma_j-1}$, $\tilde{e}_{2j} = {e}_{2\sigma_j}$, $\tilde{\alpha}_j = \alpha_{\sigma_j}$ for each $j=1, \dots, m$, as well as $\tilde{e}_j = e_j$ for all $j=2m+1, \dots, n$, then clearly this does not change $\Phi_\gamma$.
\end{enumerate}

Now generally, we have $\tilde{e}_j = Qe_j$ for all $j$, where $Q$ is some orthogonal automorphism of $T_{\gamma(0)} X$. By the preliminary considerations (1) and (2), we may assume that $\tilde{\alpha}_j = \alpha_j$ and furthermore that $\alpha_j \in (0, 1)$. Making this assumption, for $\alpha \in \{\alpha_1, \dots, \alpha_m\}$ we define subspaces
\begin{equation*}
  V_\alpha = \mathrm{span}\bigl\{ e_{2j-1}, e_{2j} \mid \alpha_j = \alpha \bigr\}  \subseteq T_{\gamma(0)} X.
\end{equation*}
and
\begin{equation*}
  V_0 = \ker\bigl([\gamma\|_1^0]^{TX} - \id\bigr).
\end{equation*}
We also set $n_\alpha = \dim(V_\alpha)$ for $\alpha \in \{0, \alpha_1, \dots, \alpha_m\}$. Then $Q$ preserves each of these subspaces $V_\alpha$ and therefore has block-diagonal form with blocks $Q_\alpha$ preserving the subspaces $V_\alpha$. Decomposing $Q$ as a product accordingly, we may now assume that all $Q_\alpha$ but one are the identity, in order to deal with each $\alpha$ separately. After possibly applying a permutation of the basis vectors, which is acceptable by (3) above, we may moreover assume that $V_\alpha$ is spanned by $e_1, \dots, e_{n_\alpha}$.

Making these simplifications, we now fix this $\alpha$ and first assume that $\alpha \notin \{0, 1/2\}$.  Then $Q_\alpha$ can be represented by a $n_\alpha \times n_\alpha$ matrix with respect to the basis, conjugation by which preserves the matrix
\begin{equation*}
R_\alpha ~~\hat{=}~~ \begin{pmatrix}
    \cos(2\pi \alpha) & -\sin(2\pi \alpha) &  & &  \\
    \sin(2\pi \alpha) & \cos(2\pi \alpha) &  & &  \\
    & & \ddots & &  \\
    & & & \cos(2\pi \alpha) & -\sin(2\pi \alpha)  \\
    & & & \sin(2\pi \alpha) & \cos(2\pi \alpha) \\
    \end{pmatrix},
\end{equation*}
i.e., $Q_\alpha$ satisfies $Q_\alpha R_\alpha Q_\alpha^* = R_\alpha$. Then necessarily $Q_\alpha \in \mathrm{U}(n_\alpha/2) \subset \mathrm{SO}(n_\alpha)$, because
\begin{equation*}
  J_\alpha \defeq \frac{R_\alpha - \cos(2\pi\alpha)\id}{\sin(2\pi\alpha)} ~~\hat{=}~~ \begin{pmatrix}
    0 & -1 &  & &  \\
    1 & 0 &  & &  \\
    & & \ddots & &  \\
    & & &0  & -1  \\
    & & & 1 & 0 \\
    \end{pmatrix}
\end{equation*}
is a complex structure on $V_\alpha$ that is preserved by conjugation with $Q_\alpha$. We now use this to show that $\tilde{\epsilon}_0 = \epsilon_0$. To this end, notice that because ${\cc}_{2j-1}{\cc}_{2j}$ and ${\cc}_{2i-1}{\cc}_{2i}$ commute in $\Cl T_{\gamma(0)} X$ for all $i, j$, we have
\begin{equation*}
  \prod_{j=1}^{k_\alpha} \bigl(\cos(\pi\alpha) + \sin(\pi\alpha) {\cc}_{2j-1}{\cc}_{2j}\bigr) = \exp \left( \pi \alpha \sum_{j=1}^{k_\alpha} {\cc}_{2j-1} {\cc}_{2j}\right) = \exp\bigl(\pi \alpha \,\cc(J_\alpha)\bigr), 
  \end{equation*}
where we identified the skew-symmetric endomorphism $J_\alpha \in \mathfrak{so}(V_\alpha)$ with an element in $\Lambda^2 V_\alpha \subset \Lambda^2 T^\prime_{\gamma(0)} X$ the usual way. Since $Q_\alpha$ preserves $J_\alpha$, this implies that $\tilde{\epsilon}_0 = \epsilon_0$.
Also, because in particular $Q_\alpha \in \mathrm{SO}(n_\alpha)$, we have $\mathrm{sign}(\tilde{e}_1, \dots, \tilde{e}_{n}) = \mathrm{sign}(e_1, \dots, e_n)$. Finally, the identity $\tilde{\Theta}_\gamma  = \Theta_\gamma$ is clear.

On the other hand, if $\alpha=0$ or $\alpha=1/2$, $Q_\alpha \in \mathrm{O}(n_\alpha)$ can be arbitrary. In the case $\alpha= 0$, it is clear that $\tilde{\epsilon}_0 = \epsilon_0$, while
\begin{equation*}
  \mathrm{sign}(\tilde{e}_1, \dots, \tilde{e}_n) = \det(Q_0) \cdot \mathrm{sign}({e}_1, \dots, {e}_n), ~~~~~~~ \tilde{\Theta}_\gamma = \det(Q_0) \cdot {\Theta}_\gamma,
\end{equation*}
where the latter sign comes from the induced permutation of $E_{2m+1}^\prime, \dots, E_n^\prime$. As $\det(Q_0)^2 = 1$, we conclude that $\tilde{\Phi}_\gamma = \Phi_\gamma$.

Finally, consider the case $\alpha = 1/2$. Then $n_\alpha = n_{1/2}$ is necessarily even, and
\begin{equation*}
  \prod_{j=1}^{\frac{n_{1/2}}{2}} \bigl(\cos(\pi /2) + \sin(\pi /2) \cc_{2j-1}\cc_{2j}\bigr) = \cc_1 \cc_2 \cdots \cc_{n_{1/2}} 
\end{equation*}
is the volume element of $\Cl(V_{1/2})$. Since the volume element of a Clifford algebra remains invariant under an orientation preserving orthogonal transformation while it changes sign if the transformation is  orientation reversing, this shows 
\begin{equation*}
\tilde{\epsilon}_0 = \det(Q_{1/2}) \cdot \epsilon_0 ~~~~~~\text{and}~~~~~~\mathrm{sign}(\tilde{e}_1, \dots, \tilde{e}_n) = \det(Q_{1/2})\cdot \mathrm{sign}({e}_1, \dots, {e}_n),
\end{equation*}
while it is clear that $\tilde{\Theta}_\gamma = \Theta_\gamma$. Hence also in this case, $\tilde{\Phi}_\gamma = \Phi_\gamma$, which concludes the proof that $\Phi_\gamma$ does not depend on the choices made to define it, hence is well-defined.

\medskip

We are left to show smoothness of the isomorphisms $\Phi_\gamma$ under change of $\gamma$. In fact, it suffices to check continuity in $\gamma$, by the following argument: Because $\Theta_\gamma$ has norm one (see \eqref{NormOfTheta}), it follows that $\Phi_\gamma$ is an isometry for each $\gamma$. Now since both $\Pf$ and $\mathcal{L}$ are real line bundles, there are precisely \emph{two} isometries $\Pf_\gamma \rightarrow \mathcal{L}_\gamma$ for each $\gamma$, and any continuous choice of such is automatically smooth.

Since $\L X$ is modelled on a Fr\'echet space, it is locally metrizable, hence it suffices to check that $\Phi_{\gamma_k} \rightarrow \Phi_\gamma$ whenever $(\gamma_k)_{k \in \N}$ is a sequence converging to $\gamma$. 
By standard results on perturbation theory of linear operators, eigenvalues and eigenspaces depend continuously on the operator (e.g., \cite[Ch.~2,~\S5]{Kato}), and it follows directly from the definition \eqref{DefinitionTheta} of $\Theta$ that $\Theta_{\gamma_n} \rightarrow \Theta_{\gamma}$ unless one of the ``eigenvalues'' $\alpha_j^{(k)}$ degenerates to an element in $\Z$ as $k \rightarrow \infty$. 

To discuss such degenerations, we need some preparations.
For $j=1, \dots, m$ and $\ell \in \Z$, we set
\begin{equation} \label{ScewBasis}
  \begin{aligned}
    V_{2j-1, \ell}(t) &\defeq \cos\bigl(2\pi(\ell+\alpha_j)t\bigr)E_{2j-1}(t) + \sin\bigl(2\pi(\ell+\alpha_j)t\bigr)E_{2j}(t)\\
    V_{2j, \ell}(t) &\defeq -\sin\bigl(2\pi(\ell+\alpha_j)t\bigr)E_{2j-1}(t) + \cos\bigl(2\pi(\ell+\alpha_j)t\bigr)E_{2j}(t).
  \end{aligned}
\end{equation}
Using \eqref{ParallelTranslatesEj324} and the fact that the parallel transport in $T^\prime X$ is also given by the matrix \eqref{MatrixOfParallelTransport} with respect to $e_1^\prime, \dots, e_n^\prime$, one checks that these fit together to continuous vector fields along $\gamma$. 
Fixing $j$ and $\ell$, the vector fields $V_{2j-1, \ell}$ and $V_{2j, \ell}$ span a $\nabla_{\dot{\gamma}}$-invariant subspace of $L^2(\T, \gamma^*TX)$, in which we have the matrix representation
\begin{equation*}
\nabla_{\dot{\gamma}} ~ \widehat{=} ~ 
  \begin{pmatrix} 
    0 & - 2\pi(\ell+\alpha_j) \\ 2\pi(\ell+\alpha_j) & 0
  \end{pmatrix}.
\end{equation*}
The Pfaffian of the endomorphism $\nabla_{\dot{\gamma}}$, restricted to this subspace (as defined in \S\ref{SectionTopDegree}), is therefore given by $- 2(\ell+\alpha_j) V^\prime_{2j-1, \ell} \wedge V^\prime_{2j, \ell}$, where $V^\prime_{2j-1, \ell}$ and $V^\prime_{2j, \ell}$ denote the sections obtained by taking the pointwise metric dual. 

After these preparations, let us return to discussing the case that $\alpha_j^{(k)}$ converges to an element of $\Z$ along the sequence $\gamma_k$. Here we may assume that $\alpha_j^{(k)} \in (0, 1)$. Suppose first that $\alpha_j^{(k)} \searrow 0$. Then using the above discussion and the formula \eqref{DefinitionTheta} for $\Theta$, we may write 
\begin{equation*}
  \Theta_{\gamma_k} = \tilde{\Theta}_{\gamma_k} \wedge \frac{\pi \alpha_j^{(k)}}{\sin(\pi \alpha_j^{(k)})} {V^{\prime (k)}_{2j-1, 0}} \wedge {V^{\prime (k)}_{2j, 0}}
\qquad
\text{and}
\qquad 
\Theta_{\gamma} = \tilde{\Theta}_\gamma \wedge E^\prime_{2j-1} \wedge E^\prime_{2j},
\end{equation*}
with orthonormal bases $e_1^{(k)}, \dots e_n^{(k)}$ of $T_{\gamma_k(0)} X$ and $e_1, \dots, e_n$ of $T_{\gamma(0)} X$ chosen in such a way that $\tilde{\Theta}_{\gamma_k} \rightarrow \tilde{\Theta}_\gamma$. Now since $\alpha_j^{(k)} \searrow 0$ and consequently $V_{2j-1}^{\prime} \rightarrow E_{2j-1}$, $V_{2j}^\prime \rightarrow E_{2j}$, we also obtain $\Theta_{\gamma_k} \rightarrow \Theta_\gamma$.

On the other hand, if $\alpha_j^{(k)} \nearrow 1$, we write 
\begin{equation*}
  \Theta_{\gamma_k} = \tilde{\Theta}_{\gamma_k} \wedge \frac{\pi (\alpha_j^{(k)}-1)}{\sin(\pi \alpha_j^{(k)})} {V^{\prime (k)}_{2j-1, -1}} \wedge {V^{\prime (k)}_{2j, -1}}
\qquad
\text{and}
\qquad 
\Theta_{\gamma} = \tilde{\Theta}_\gamma \wedge E^\prime_{2j-1} \wedge E^\prime_{2j}
\end{equation*}
Hence $\Theta_{\gamma_k} \to - \Theta_\gamma$. However, in the limit $\alpha_j^{(k)} \to 1$, also the sign $\epsilon_0$ defined by \eqref{Eq:FormulaSpinorParallelTransport} changes to $-\epsilon_0$, which cancels the extra minus sign. Hence $\Phi_{\gamma_k} \to \Phi_\gamma$, as desired.
\end{proof}

\begin{remark} \label{RemarkMapsToEachOther}
It is now easy to see that $\Phi$ maps the canonical sections of $\mathcal{L}$ respectively $\Pf$ to each other. Since both of these sections vanish if $m < n/2$ (in particular if $n$ is odd), we may assume that $m = n/2$. Plugging the formula \eqref{FormulaSupertrace} for the supertrace of the parallel transport into $\Phi$ (see \eqref{DefinitionPhiGamma}) and comparing with \eqref{DefinitionPfaffian}, we obtain 
\begin{equation*}
\Phi_\gamma\bigl(\str \bigl([\gamma\|_1^0]^{\Sigma}\bigr)\bigr) = \pf(\nabla_{\dot{\gamma}}), 
\end{equation*}
as desired.
\end{remark}

\section{Definition of the Top Degree Functional} \label{SectionTopDegree}

In this section, we define the top degree functional on a suitable subspace of $L_{\mathrm{alt}}^N(T_\gamma \L X, \R)$ taking values in the Pfaffian line bundle $\Pf$ defined in \S\ref{SectionLineBundles}. We start by discussing the top-degree functional on $L_{\mathrm{alt}}^N(V, \R) = \Lambda^N V^\prime$ for $V$ a finite-dimensional vector space, and then generalize this to the loop space by analogy.

\paragraph{Pfaffians.} Given a skew-symmetric real matrix $X = (x_{ab})_{1 \leq a, b \leq N}$, its Pfaffian is defined by
\begin{equation} \label{Eq:DefinitionPfaffian}
  \pf(X) \defeq \frac{1}{2^{N/2} \left(N/2\right)!} \sum_{\sigma \in S_N} \mathrm{sgn}(\sigma) \prod_{a=1}^{N/2} x_{\sigma_{2a-1}, \sigma_{2a}}
\end{equation}
if $N$ is even, while by convention, $\pf(X) = 0$ for $N$ odd. In particular, for $2\times 2$ matrices, one has
\begin{equation} \label{PfaffianTwoByTwoMatrix}
 \pf \begin{pmatrix} 0 & \lambda \\ -\lambda & 0 \end{pmatrix} = \lambda.
\end{equation}
More generally, for any $d \times d$-matrix $Y$, we have
\begin{equation} \label{BlockPfaffianDeterminant}
  \pf\begin{pmatrix} 0 & Y \\ -Y^t & 0 \end{pmatrix} = (-1)^{d(d-1)/2} \det(Y).
\end{equation}
The formula remains valid also for a non-square matrix $Y$, in which both sides are zero.
The above formula can be easily proved by induction, using the development formula for the determinant and the similar formula
\begin{equation} \label{PfaffianRecursion}
  \pf(X) = \sum_{a=2}^N (-1)^a x_{1a} \cdot \pf(X_{\hat{1}\hat{a}}),
\end{equation}
for the Pfaffian,
where $X_{\hat{1}\hat{a}}$ denotes the matrix obtained from $X$ by removing both the first and $a$-th row and column. We will also use that the Pfaffian has the property
\begin{equation} \label{PfaffianPermutation}
 \pf 
 \begin{pmatrix} 
 x_{\sigma_1 \sigma_1} & \cdots & x_{\sigma_1 \sigma_N} \\ 
 \vdots & & \vdots \\
 x_{\sigma_N \sigma_1} & \cdots & x_{\sigma_N \sigma_N} 
 \end{pmatrix} 
 = 
 \sgn(\sigma) \cdot
 \pf
 \begin{pmatrix} 
 x_{1 1} & \cdots & x_{1 N} \\ 
 \vdots & & \vdots \\
 x_{N 1} & \cdots & x_{N N} 
 \end{pmatrix} 
\end{equation}
for any permutation $\sigma \in S_N$. This follows from the more general formula 
\begin{equation} \label{PfaffianConjugation}
\pf(Y X Y^t) = \det(Y) \pf(X), 
\end{equation}
applied to the permutation matrix $Y$ associated to $\sigma$.

\paragraph{The finite-dimensional top degree functional.} Now if $A$ is a skew-symmetric endomorphism of a (finite-dimensional) abstract Euclidean vector space $V$, one needs the additional datum of an {\em orientation} of $V$ to define the Pfaffian $\pf(A)$. In this case, the Pfaffian is given by \eqref{Eq:DefinitionPfaffian} if $X$ is any matrix representation of $A$ with respect to an {\em oriented} orthonormal basis. This is independent of the choice of such orthonormal basis, but changing the orientation of $V$ will change the sign of $\pf(A)$.

Without the datum of an orientation, one can still define the Pfaffian of $A$, but then it will be an element of the line $\Lambda^{\mathrm{top}} V^\prime$. Namely,  $A$ induces a two-form $\omega_A \in \Lambda^2 V^\prime$ defined by 
\begin{equation} \label{DefinitionOmegaFromA}
\omega_A(v, w) \defeq \langle v, Aw\rangle.
\end{equation}
We can then exponentiate $\omega_A$ in the exterior algebra $\Lambda V^\prime$, to obtain the element $e^{\omega_A}$; the Pfaffian of $A$ is then given by
\begin{equation*}
  \pf(A) \defeq (-1)^{\dim(V)/2} [e^{-\omega_A}]_{\mathrm{top}} \in \Lambda^{\mathrm{top}}V^\prime,
\end{equation*}
the top degree component of the mixed degree differential form $e^{-\omega_A}$. 
An orientation of $V$ provides a volume form $\mathrm{vol} \in \Lambda^{\mathrm{top}} V^\prime$, whence a canonical isomorphism $\Lambda^{\mathrm{top}} V^\prime \cong \R$. Under this isomorphism, $\pf(A)$ is carried to the number defined by \eqref{Eq:DefinitionPfaffian} using the orientation, as above.

The Pfaffian $\pf(A)$ is nonzero only if $A$ is invertible. In general, we have 
\begin{equation*}
e^{-\omega_A} \in \Lambda (\ker(A)_\perp) \subseteq \Lambda V^\prime,
\end{equation*}
where for a subspace $U \subset V$, $U_\perp = \{\vartheta \in V^\prime \mid \forall v \in U:\vartheta(v) = 0 \}$ denotes the annihilator of $U$. The \emph{reduced Pfaffian} of $A$ is then
\begin{equation*}
\pf^\prime(A) = (-1)^{(\dim(V)-\ker(A))/2} [e^{-\omega_A}]_{\mathrm{top}} \in \Lambda^{\mathrm{top}}(\ker(A)_\perp).
\end{equation*}

The following is an extension of Prop.~1 in \cite{Lott}; also see \S1 in \cite{MR836726}. It makes use of the ``Green endomorphism'' for $A^\prime$, which is the skew-symmetric endomorphism of $V^\prime$ given by
\begin{equation} \label{DefinitionGreenEndomorphism}
 G =  \begin{pmatrix} 0 & 0 \\ 0 & (A^\prime)^{-1} \end{pmatrix}
\end{equation}
with respect to the splitting 
\begin{equation*}
V^\prime =  \ker(A^\prime) \oplus \ker(A^\prime)^\perp =   (\ker(A)^\perp)_\perp \oplus \ker(A)_\perp.
\end{equation*}

\begin{proposition} \label{Prop:FermionicWick}
Denote $d= \dim\ker(A) = \dim\ker(A^\prime)$ and let $e^\prime_1, \dots, e^\prime_d$ be an orthonormal basis of $\ker(A^\prime) \subseteq V^\prime$. Given $\vartheta_1, \dots, \vartheta_{N} \in V^\prime$, form the matrix
  \begin{equation} \label{XMatrixFinDim}
    X = \begin{pmatrix}
0 & \cdots & 0 & \vline &   \langle e^\prime_d, \vartheta_1 \rangle & \cdots & \langle e^\prime_d, \vartheta_N \rangle \\
\vdots &  & \vdots &  \vline &   \vdots & & \vdots \\
0 & \cdots & 0 & \vline &    \langle e^\prime_1, \vartheta_1 \rangle & \cdots & \langle e^\prime_1, \vartheta_N \rangle \\
& & & \vline & & \\
\hline& & & \vline & &  \\
-\langle \vartheta_1, e^\prime_d \rangle & \cdots & -\langle \vartheta_1, e^\prime_1 \rangle &  \vline &   \langle \vartheta_1, G \vartheta_1 \rangle & \cdots & \langle \vartheta_1, G \vartheta_N \rangle \\
\vdots &  & \vdots &  \vline &   \vdots & & \vdots \\
 -\langle \vartheta_N, e^\prime_d \rangle & \cdots & -\langle \vartheta_N, e^\prime_1 \rangle &  \vline &  \langle \vartheta_N, G \vartheta_1 \rangle & \cdots & \langle \vartheta_N, G \vartheta_N \rangle
    \end{pmatrix}.
  \end{equation}
  Then we have the formula
  \begin{equation}  \label{TopDegreeFiniteDimensions}
    \bigl[e^{-\omega_A} \wedge \vartheta_1 \wedge \cdots \wedge \vartheta_N\bigr]_{\mathrm{top}} = (-1)^{(N-\dim(V))/2} \, \pf\bigl( X \bigr) \cdot \pf^\prime(A) \wedge e^\prime_1 \wedge \cdots \wedge e^\prime_d,
  \end{equation}
  where $\pf(A^\prime)  \in\Lambda^{\mathrm{top}}(\ker(A^\prime)_\perp)$ is the reduced Pfaffian defined above.
\end{proposition}

In \eqref{TopDegreeFiniteDimensions}, we identify $\pf^\prime(A) \in\Lambda^{\mathrm{top}}(\ker(A^\prime)_\perp) \subseteq \Lambda^{\mathrm{top}- d} V^\prime$, so that the wedge product with $e^\prime_1 \wedge \dots \wedge e^\prime_d$ gives an element in $\Lambda^{\mathrm{top}} V^\prime$. 

In particular, if $A$ is invertible, we have $\pf^\prime(A) = \pf(A)$ (as elements in $\Lambda^{\mathrm{top}} V^\prime$) and
\begin{equation} \label{FormulaAInvertible}
\bigl[e^{-\omega_A} \wedge \vartheta_1 \wedge \cdots \wedge \vartheta_N\bigr]_{\mathrm{top}} = (-1)^{(N-\dim(V))/2}\pf( X ) \cdot \pf(A).
\end{equation}

\begin{proof}
Both sides of \eqref{TopDegreeFiniteDimensions} define a multilinear, alternating map $V^\prime \times \cdots \times V^\prime \to \Lambda^{\mathrm{top}} V^\prime$. For the left hand side, this is so by properties of the wedge product; for the right hand side this follows from the property \eqref{PfaffianPermutation} of the Pfaffian.

In order to verify \eqref{TopDegreeFiniteDimensions}, we therefore may assume that for some $M$, 
\begin{equation} \label{AssumptionVartheta}
\vartheta_1, \dots, \vartheta_M \in \ker(A^\prime) \qquad \text{and} \qquad \vartheta_{M+1}, \dots, \vartheta_N \in \ker(A^\prime)^\perp. 
\end{equation}
Then $[e^{-\omega_A}\wedge \vartheta_1 \wedge \cdots \wedge \vartheta_N]_{\mathrm{top}}$ can be non-zero only if $M = d$, in which case we have
\begin{equation*}
 \bigl[e^{-\omega_A}\wedge \vartheta_1 \wedge \cdots \wedge \vartheta_N\bigr]_{\mathrm{top}} = \bigl[\vartheta_1 \wedge \cdots \wedge \vartheta_{d}\bigr]_{\mathrm{top}} \wedge \bigl[ e^{-\omega_A} \wedge \vartheta_{d+1}\wedge \cdots \wedge \vartheta_{N}\bigr]_{\mathrm{top}},
 \end{equation*} 
where on the right hand side, we identify $\Lambda^{\mathrm{top}}\ker(A^\prime) \wedge \Lambda^{\mathrm{top}} \ker(A^\prime)^\perp = \Lambda^{\mathrm{top}} V^\prime$. The left factor can be written as
\begin{equation} \label{KernelContribution}
  \bigl[\vartheta_1 \wedge \cdots \wedge \vartheta_{d}\bigr]_{\mathrm{top}} 
  = \det \begin{pmatrix} \langle e^\prime_1, \vartheta_1 \rangle & \cdots & \langle e^\prime_1, \vartheta_d \rangle \\
  \vdots & & \vdots \\
  \langle e^\prime_d, \vartheta_1 \rangle & \cdots & \langle e^\prime_d, \vartheta_d \rangle
\end{pmatrix} \cdot e^\prime_1 \wedge \cdots \wedge e^\prime_d.
\end{equation}
On the other hand, under assumption \eqref{AssumptionVartheta}, we have $\langle e^\prime_j, \vartheta_a \rangle = 0$ for each $j$ whenever $a \geq M +1$ and by the definition \eqref{DefinitionGreenEndomorphism} of $G$, we have $\langle \vartheta_a, G \vartheta_b \rangle = 0$ whenever either $a \leq M$ or $b \leq M$. We obtain that the matrix $X$ takes the form
  \begin{footnotesize}
  \begin{equation*}
 \begin{pmatrix}
0 & \cdots & 0 & \vline &   \langle e^\prime_d, \vartheta_1 \rangle & \cdots & \langle e^\prime_d, \vartheta_M \rangle & \vline &  0 & \cdots & 0 \\
\vdots &  & \vdots &  \vline &   \vdots & & \vdots & \vline & \vdots & & \vdots  \\
0 & \cdots & 0 & \vline &    \langle e^\prime_1, \vartheta_1 \rangle & \cdots & \langle e^\prime_1, \vartheta_M \rangle & \vline & 0 & \cdots & 0 \\
& & & \vline & & & & \vline & & & \\
\hline& & & \vline & & & & \vline & & &  \\
-\langle \vartheta_1, e^\prime_d \rangle & \cdots & -\langle \vartheta_1, e^\prime_1 \rangle &  \vline &   0 & \cdots & 0 & \vline & 0 & \cdots & 0 \\
\vdots &  & \vdots &  \vline &   \vdots & & \vdots & \vline & \vdots &  & \vdots  \\
 -\langle \vartheta_M, e^\prime_d \rangle & \cdots & -\langle \vartheta_M, e^\prime_1 \rangle &  \vline &  0 & \cdots & 0 & \vline & 0 & \cdots & 0  \\
 & & & \vline & & & &\vline & &  \\
\hline& & & \vline & &  & & \vline & &  \\
0 & \cdots & 0 & \vline & 0 & \cdots & 0 &  \vline &   \langle \vartheta_{M+1}, G \vartheta_{M+1}\rangle & \cdots & \langle \vartheta_{M+1}, G \vartheta_N \rangle \\
\vdots &  & \vdots &  \vline &\vdots &  & \vdots &  \vline &   \vdots & & \vdots \\
 0 & \cdots & 0 &  \vline & 0 & \cdots & 0 &  \vline &  \langle \vartheta_N, G \vartheta_{M+1} \rangle & \cdots & \langle \vartheta_N, G \vartheta_N \rangle
    \end{pmatrix}.
  \end{equation*}  
\end{footnotesize}
Hence $X$ is in blockdiagonal form and hence $\pf(X)$ is the product of the Pfaffian of the individual blocks. Using \eqref{BlockPfaffianDeterminant} (and the remark below this formula), we see that the Pfaffian of the upper left block vanishes und less $M =d$, and in this case, $\pf(X)$ is given by
\begin{equation*}
(-1)^{d(d-1)/2} \det \begin{pmatrix} \langle e^\prime_d, \vartheta_1 \rangle & \cdots & \langle e^\prime_d, \vartheta_d \rangle \\
  \vdots & & \vdots \\
  \langle e^\prime_1, \vartheta_1 \rangle & \cdots & \langle e^\prime_1, \vartheta_d \rangle
\end{pmatrix} \cdot 
\pf \begin{pmatrix}  \langle \vartheta_{d+1}, G \vartheta_{d+1}\rangle & \cdots & \langle \vartheta_{d+1}, G \vartheta_N \rangle \\
\vdots & & \vdots \\
\langle \vartheta_N, G \vartheta_{d+1} \rangle & \cdots & \langle \vartheta_N, G \vartheta_N \rangle
\end{pmatrix}.
\end{equation*}
Observing that $(-1)^{d(d-1)/2}$ is precisely the sign of the permutation in $S_d$ sending $(1, \dots, d)$ to $(d, \dots, 1)$, and comparing with \eqref{KernelContribution}, it remains to verify the identity
\begin{equation*}
  \bigl[e^{-\omega_A} \wedge \vartheta_{d+1} \wedge \cdots \wedge \vartheta_N\bigr]_{\mathrm{top}} = 
  (-1)^{(N-\dim(V))/2}
  \pf \begin{pmatrix}  \langle \vartheta_{d+1}, G \vartheta_{d+1}\rangle & \cdots & \langle \vartheta_{d+1}, G \vartheta_N \rangle \\
\vdots & & \vdots \\
\langle \vartheta_N, G \vartheta_{d+1} \rangle & \cdots & \langle \vartheta_N, G \vartheta_N \rangle
\end{pmatrix} 
 \pf(A^\prime)
\end{equation*}
in $\Lambda^{\mathrm{top}}(\ker(A^\prime)^\perp)$. In other words, we reduced to the case where  $A$ is invertible, in which case we need to establish the formula \eqref{FormulaAInvertible}. 

By choosing a basis in which  $A$ is in real Jordan normal form (i.e., block-diagonal with $2 \times 2$ skew-symmetric blocks), one readily reduces to the case $\dim(V) = 2$, where
\begin{equation*}
 A = \begin{pmatrix} 0 & \lambda \\ - \lambda & 0 \end{pmatrix} \qquad \text{and} \qquad A^{-1} = \begin{pmatrix} 0 & -\tfrac{1}{\lambda} \\ \tfrac{1}{\lambda} & 0 \end{pmatrix}
\end{equation*}
with respect to some orthonormal basis $e_1$, $e_2$ of $V$. Then ${\omega_A} = \lambda e_1^\prime \wedge e_2^\prime$, where $e_1^\prime$, $e_2^\prime$ is the corresponding dual basis. The cases $N=0$, $N\geq 4$ and $N$ odd are trivial, so the only case left to check is $N=2$. Here we get
\begin{equation} \label{2DcalcInProof}
  [e^{-\omega_A} \wedge e_1^\prime \wedge e_2^\prime]_{\mathrm{top}} = \bigl[(1-\lambda e_1^\prime \wedge e_2^\prime) \wedge e_1^\prime \wedge e_2^\prime\bigr]_{\mathrm{top}} = e_1^\prime \wedge e_2^\prime.
\end{equation}
On the other hand, we have $G e_1^\prime = -\frac{1}{\lambda} e_2^\prime$ and $G e_2^\prime = \frac{1}{\lambda} e_1^\prime$, hence
\begin{equation*}
  \pf(S) = \pf\begin{pmatrix} 0 & \langle e_1^\prime, G e_2^\prime\rangle \\ \langle e_2^\prime, G e_1^\prime \rangle & 0 \end{pmatrix} = \pf \begin{pmatrix} 0 & \tfrac{1}{\lambda} \\ - \tfrac{1}{\lambda} & 0  \end{pmatrix} = \frac{1}{\lambda}
\end{equation*}
Since $\pf(A) = \lambda e_1^\prime \wedge e_2^\prime$, we get $\pf(S) \cdot \pf(A) = e_1^\prime \wedge e_2^\prime$. This coincides with \eqref{2DcalcInProof}, as desired, and finishes the proof. 
\end{proof}

\paragraph{Application to the loop space.} Let now $X$ be a Riemannian spin manifold of dimension $n$. For a smooth loop $\gamma \in \L X$, we let the tangent space $T_\gamma \L X = C^\infty(\T, \gamma^* TX)$ play the role of $V$ in the above discussion. Of course, this space is not finite-dimensional, but it has a natural scalar product, namely the $L^2$ scalar product defined in \eqref{L2ScalarProduct}, which turns it into an infinite-dimensional Euclidean space (or pre-Hilbert space). Observe that the canonical 2-form defined in \eqref{CanonicalTwoForm} can be written as
\begin{equation*}
\omega(V, W) = \langle V, \nabla_{\dot{\gamma}} W\rangle_{L^2}, \qquad V, W \in T_\gamma \L X.
\end{equation*}
Hence comparing with \eqref{DefinitionOmegaFromA}, we see that in this situation, the operator $A$ from above is precisely the covariant derivative $\nabla_{\dot{\gamma}}$. We can equivalently consider the convariant derivative $\nabla^\prime_{\dot{\gamma}}$ acting on $\gamma^*T^\prime X$, which will play the role of $A^\prime$; it has a finite-dimensional kernel and its  Green operator $G$ is a well-defined bounded operator on $L^2(\T, \gamma^*T^\prime X)$, defined by
 \begin{equation} \label{GreenOperatorNabla}
 G =  \begin{pmatrix} 0 & 0 \\ 0 & (\nabla_{\dot{\gamma}}^\prime)^{-1} \end{pmatrix}
 \end{equation}
 with respect to the direct sum decomposition of $L^2(\T, \gamma^*T^\prime X)$ into the kernel of $\nabla_{\dot{\gamma}}^\prime$ (which is the space of parallel one-forms along $\gamma$) and its orthogonal complement.
   Inspecting the finite-dimensional formula \eqref{TopDegreeFiniteDimensions} above, we observe that it can be carried over in a straightforward way to the infinite-dimensional setup, yielding the following definition.

\begin{definition}[The top degree] \label{DefFinalTopDegree}
Let $d = \dim\ker(\nabla_{\dot{\gamma}}^\prime)$ and let $E_1, \dots, E_d$ be an orthonormal basis of $\ker(\nabla_{\dot{\gamma}})$, with dual basis $E^\prime_1, \dots, E^\prime_d \in \ker(\nabla_{\dot{\gamma}}^\prime)$.
Given $\theta_1, \dots, \theta_N \in L^2(\T, \gamma^*T^\prime X) \subset T^\prime_\gamma \L X$, consider the matrix
  \begin{equation} \label{MatrixX}
    X = \begin{pmatrix}
0 & \cdots & 0 & \vline &   \langle E^\prime_d, \theta_1 \rangle_{L^2} & \cdots & \langle E^\prime_d, \theta_N \rangle_{L^2} \\
\vdots &  & \vdots &  \vline &   \vdots & & \vdots \\
0 & \cdots & 0 & \vline &    \langle E^\prime_1, \theta_1 \rangle_{L^2} & \cdots & \langle E^\prime_1, \theta_N \rangle_{L^2} \\
& & & \vline & & \\
\hline& & & \vline & &  \\
-\langle \theta_1, E^\prime_d \rangle_{L^2} & \cdots & -\langle \theta_1, E^\prime_1 \rangle_{L^2} &  \vline &   \langle \theta_1, G \theta_1 \rangle_{L^2} & \cdots & \langle \theta_1, G \theta_N \rangle_{L^2} \\
\vdots &  & \vdots &  \vline &   \vdots & & \vdots \\
 -\langle \theta_N, E^\prime_d \rangle_{L^2} & \cdots & -\langle \theta_N, E^\prime_1 \rangle_{L^2} &  \vline &  \langle \theta_N, G \theta_1 \rangle_{L^2} & \cdots & \langle \theta_N, G \theta_N \rangle_{L^2}
    \end{pmatrix}.
 \end{equation}
The {\em top degree} is the element of $\mathrm{Pf}_\gamma$ defined by the formula
\begin{equation} \label{TopDegreeInfiniteDimensions}
[e^{-\omega} \wedge \theta_1 \wedge \dots \wedge \theta_N]_{\mathrm{top}} \defeq 
(-1)^{(N-n)/2} \cdot \pf(X) \cdot { \pf^\prime (\nabla_{\dot{\gamma}})}  \wedge E^\prime_1 \wedge \cdots \wedge E^\prime_d,
\end{equation}
where   ${ \pf^\prime(\nabla_{\dot{\gamma}})} \in \Pf_\gamma \otimes(\Lambda^{\mathrm{top}} \ker(\nabla_{\dot{\gamma}}))^\prime$ is the reduced Pfaffian, defined in \eqref{DefReducedPfaffian}.
\end{definition}

It is easy to check that the right hand side of \eqref{TopDegreeInfiniteDimensions} does not depend on the choice of basis $E^\prime_1, \dots, E^\prime_d$. However, it is not clear from this definition that the functional $[e^{-\omega} \wedge -]_{\mathrm{top}}$ depends continuously or smoothly on $\gamma$. This is a consequence of Thm.~\ref{ThmMainTheoremLvsPf} in the next section.

\section{The Main Theorem} \label{SectionMainTheorem}

This section is dedicated to proving another formula for the top degree map  from Def.~\ref{DefFinalTopDegree} that connects it to spin geometry. We begin with the following definition,  which generalizes formula \eqref{FormulaqIntro} from the introduction in the case that $M$ is not a spin manifold.

\begin{definition}
Let $X$ be an oriented Riemannian manifold and $\gamma \in \L X$. Given a spin structure $P$ on $\gamma^*TX$ with associated spinor bundle $\Sigma_P$, and elements $\theta_1, \dots, \theta_N \in L^2(\T, \gamma^*T^\prime X)$, define 
\begin{equation} \label{MapQInvariant}
\Ber_P(\theta_1, \dots, \theta_N) \defeq 2^{-N/2}  \sum_{\sigma \in S_N} \mathrm{sgn}(\sigma) \int_{\Delta_N}\str \left([\gamma\|_{\tau_1}^0]^{\Sigma_P} \prod_{a=1}^N \cc\bigl(\theta_{\sigma_a}(\tau_a)\bigr) [\gamma\|_{\tau_{a+1}}^{\tau_a}]^{\Sigma_P}\right)\dd \tau,
\end{equation}
where by convention $\tau_{N+1} = 1$. It is clear that $\Ber$ is alternating, hence descends to a linear functional on $\Lambda^N L^2(\T, \gamma^* T^\prime X)$.
Setting $\Ber(\theta_1, \dots, \theta_N) = [P, \Ber_P(\theta_1, \dots, \theta_N)]  \in \mathcal{L}_\gamma$, we obtain a well-defined linear functional
\begin{equation*}
  \Ber: \Lambda^N L^2(\T, \gamma^* T^\prime X) \longrightarrow \mathcal{L}_\gamma.
\end{equation*}
\end{definition}

Observe that in formula \eqref{MapQInvariant}, the term in the brackets commutes with right multiplication in $\Cl_n$, hence is contained in $\End_{\Cl_n}(\Sigma_{\gamma(0)}) \cong \Cl(T_x X)$. This allows  to take its super trace as in \eqref{DefinitionSupertrace}. Notice also that $\Ber_P(\theta)$ is always zero unless $n$ and $N$ have the same parity. The main result of this section is now the following.

\begin{theorem} \label{ThmMainTheoremLvsPf}
The line bundle isomorphism $\Phi: \mathcal{L} \rightarrow \Pf$ constructed in {\normalfont \S\ref{SectionLineBundles}} sends the functional $\Ber$ to the top degree functional. In other words, for all $\theta_1, \dots, \theta_N \in L^2(\T, \gamma^*T^\prime X)$, we have 
\begin{equation*}
\Phi \bigl(\Ber(\theta_1, \dots, \theta_N)\bigr) = [e^{-\omega} \wedge \theta_1 \wedge \dots \wedge \theta_N]_{\mathrm{top}} \in \Pf_\gamma.
\end{equation*}
\end{theorem}

\begin{remark} 
The parallel transport along $\gamma$ depends smoothly on $\gamma$, which implies that the  functional $\Ber(-)$ in fact depends smoothly on $\gamma$ as well. Therefore, Thm.~\ref{ThmMainTheoremLvsPf} in particular shows that the functional $[e^{-\omega} \wedge -]_{\mathrm{top}}$ is a smooth as well, a fact that is not at all obvious from Def.~\ref{DefFinalTopDegree}.
\end{remark}

We being by observing that the  functional $\Ber$ in fact has an integral kernel. To this end, write 
\begin{equation} \label{SlicedTorus}
\T^N_\circ = \bigl\{ \tau = (\tau_1, \dots, \tau_N) \in \T^N ~\bigl|~ \tau_a \neq \tau_b ~\text{for all} ~ a \neq b \bigr\}
\end{equation}
and after fixing a path $\gamma \in \L X$ and a spin structure $P$ on $\gamma^*TX$, define an element $\ber_P \in C^\infty(\T^N_\circ, \gamma^*T X^{\boxtimes N})$ by the formula
\begin{equation} \label{DefinitionDiracDensity}
  \ber_P(\tau)[\xi_1, \dots, \xi_N] \defeq  2^{-N/2}\,\mathrm{sgn}(\sigma) \str \left([\gamma\|_{\tau_1}^0]^{\Sigma} \prod_{a=1}^N \cc(\xi_{\sigma_a}) [\gamma\|_{\tau_{a+1}}^{\tau_a}]^{\Sigma_P}\right)
\end{equation}
for $\tau = (\tau_1, \dots, \tau_N)\in \T^N_\circ$ and  $\xi_a \in T_{\gamma(\tau_a)}'X$, where $\sigma \in S_N$ is the permutation such that $\tau_{\sigma_1} < \dots < \tau_{\sigma_N}$; this permutation exists and is unique by the requirement that $\tau \in \T_\circ^N \subset \T^N$, i.e., the $\tau_a$ are mutually distinct. 

The definition $\ber(\tau)[\xi_1, \dots, \xi_N] = [P, \ber_P(\tau)[\xi_1, \dots, \xi_N]] \in \mathcal{L}_\gamma$ is then independent of the choice spin structure $P$, hence we obtain an element 
\begin{equation*}
\ber \in C^\infty(\T^N_\circ, \gamma^*T X^{\boxtimes N})^{S_N} \otimes \mathcal{L}_\gamma, 
\end{equation*}
the subspace invariant under the action of $S_N$ given by \eqref{SNActionOnSection}. Now notice that this element is chosen in such a way that
\begin{equation} \label{qAndD}
  \Ber(\theta_1,  \cdots, \theta_N) = \int_{\T^N} \ber(\tau)\bigl[\theta_1(\tau_1), \dots, \theta_N(\tau_N)\bigr] \dd \tau;
\end{equation}
in other words, the functional $\Ber$ is just given by pairing with the element $\ber$ defined above in the $L^2$ sense. A similar result is true for the top degree component, as the following lemma shows.

\begin{lemma} \label{PfDensity}
Analogous to {\normalfont \eqref{qAndD}}, the top degree functional from {\normalfont Def.~\ref{DefFinalTopDegree}} admits an integral kernel. In other words, there exists a section
\begin{equation*}
\ber^\prime \in C^\infty(\T^N_\circ, \gamma^*TX^{\boxtimes N})^{S_N} \otimes \Pf_\gamma.
\end{equation*}
such that the top degree functional is given by the integral
\begin{equation} \label{TopandD}
  [e^{-\omega} \wedge \theta_1 \wedge \cdots \wedge \theta_N]_{\mathrm{top}}
  = \int_{\T^N} \ber^\prime(\tau)\bigl[\theta_1(\tau_1), \dots, \theta_N(\tau_N)\bigr] \dd \tau.
\end{equation}
This section is given as follows. Choose an orthonormal basis $e_1, \dots, e_n$ of $T_{\gamma(0)} X$ such that with respect to this basis, the parallel transport $[\gamma\|_1^0]^{T X}$ is given by the matrix \eqref{MatrixOfParallelTransport} for numbers $\alpha_1, \dots, \alpha_m \notin \Z$. These choices determine an element $\Theta_\gamma \in \Pf_\gamma$, given by \eqref{DefinitionTheta}. 
Let $e_1^\prime, \dots, e_n^\prime$ be the corresponding dual basis and let 
\begin{equation} \label{ParallelTranslatesEj}
E_j(t) = [\gamma\|_t^1]^{TX}e_j, \qquad \text{respectively} \qquad E_j^\prime(t) = [\gamma\|_t^1]^{T^\prime X} e_j^\prime
\end{equation}
be the corresponding parallel translates. For $\tau \in \T^N_\circ$, define the matrix $X(\tau)$ by
  \begin{footnotesize}
    \begin{equation} \label{XMatrixInLemma}
    X(\tau) = \begin{pmatrix}
0 & \cdots & 0 & \vline &  \langle E_n^\prime(\tau_1), \xi_1\rangle & \cdots  & \langle E_n^\prime(\tau_N), \xi_N \rangle \\
\vdots &  & \vdots &  \vline &   \vdots & & \vdots \\
0 & \cdots & 0 & \vline &    \langle E_{2m+1}^\prime(\tau_1), \xi_1 \rangle  & \cdots & \langle E_{2m+1}(\tau_N), \xi_N \rangle\\
& & & \vline & & \\
\hline& & & \vline & &  \\
-\langle E_n^\prime(\tau_1), \xi_1 \rangle & \cdots & -\langle E_{2m+1}^\prime(\tau_1), \xi_1 \rangle &  \vline & \bigl\langle \xi_1, g(\tau_1, \tau_1)\xi_1\bigr\rangle & \cdots & \bigl\langle \xi_1, g(\tau_1, \tau_N)\xi_N\bigr\rangle \\
\vdots &   & \vdots &\vline & \vdots & & \vdots \\
 -\langle E_n^\prime(\tau_N), \xi_N \rangle & \cdots & -\langle E^\prime_{2m+1}(\tau_1), \xi_1 \rangle &  \vline &  \bigl\langle \xi_N, g(\tau_N, \tau_1)\xi_1\bigr\rangle & \cdots & \bigl\langle \xi_N, g(\tau_N, \tau_N)\xi_{N}\bigr\rangle 
    \end{pmatrix},
  \end{equation}
    \end{footnotesize}
  where $g(t, s)$ is the integral kernel of the Green's operator $G$ for $\nabla_{\dot{\gamma}}^\prime$, defined in {\normalfont \eqref{GreenOperatorNabla}}.
In these terms,
\begin{equation} \label{Formulaberprime}
\ber^\prime(\tau)[\xi_1, \dots, \xi_N] = (-1)^{(N-n)/2} \cdot \prod_{j=1}^m \bigl(- 2 \sin(\alpha_j)\bigr) \cdot \pf\bigl(X(\tau)\bigr) \cdot \Theta_\gamma
\end{equation}
\end{lemma}

We remark here that since $\nabla_{\dot{\gamma}}$ is an elliptic differential operator, its Green's operator $G$ is a pseudodifferential operator, hence the corresponding kernel $g$ is smooth on $\T^2_\circ \subset \T^2$; hence $\ber^\prime$ is indeed smooth on $\T^N_\circ$. On the diagonal, we use the convention $g(t, t) = 0$; hence the lower right block of the matrix $X$ from \eqref{XMatrixInLemma} is given explicitly by
\begin{equation} \label{MatrixNoKernel}
\begin{pmatrix}
0 &  \bigl\langle \xi_1, g(\tau_1, \tau_2)\xi_2\bigr\rangle & \cdots & \bigl\langle \xi_1, g(\tau_1, \tau_N)\xi_N\bigr\rangle \\
 \bigl\langle \xi_2, g(\tau_2, \tau_1)\xi_1\bigr\rangle & 0 & \ddots & \vdots \\
   \vdots & \ddots & \ddots & \bigl\langle \xi_{N-1}, g(\tau_1, \tau_N)\xi_N\bigr\rangle  \\
  \bigl\langle \xi_N, g(\tau_N, \tau_1)\xi_1\bigr\rangle & \cdots & \bigl\langle \xi_N, g(\tau_N, \tau_{N-1})\xi_{N-1}\bigr\rangle & 0
    \end{pmatrix}.
\end{equation}

\begin{proof} 
We begin by rewriting the formula of Def.~\ref{DefFinalTopDegree} in terms of the element $\Theta_\gamma$, defined in \eqref{DefinitionTheta}. To this end, choose  $b>0$ in such a way that $4\pi^2 b^2$ is smaller than the smallest positive eigenvalue of $-\nabla_{\dot{\gamma}}^2$. Then, since the sections $E_{2m+1}, \dots, E_n$ of $\gamma^*TX$ form an orthonormal basis of $\ker(\nabla_{\dot{\gamma}})$, we have
\begin{equation*}
  \pf^\prime(\nabla_{\dot{\gamma}}) \wedge E_{2m+1} \wedge \cdots \wedge E_n = \prod_{j=1}^m \bigl(-2\sin(\pi \alpha_j)\bigr) \cdot \Theta_{\gamma}.
\end{equation*}
Now given $\theta_1, \dots, \theta_N$, let $X$ denote the corresponding matrix \eqref{MatrixX} and for $\tau \in \T^N$,  $X(\tau)$ denotes the matrix given by \eqref{XMatrixInLemma} with $\xi_a = \theta_a(\tau_a)$. Then by the definition \eqref{TopDegreeInfiniteDimensions}, 
\begin{equation*}
 [e^{-\omega} \wedge \theta_1 \wedge \cdots \wedge \theta_N]_{\mathrm{top}}
= (-1)^{(N-n)/2} \cdot \pf(X) \cdot \prod_{j=1}^m \bigl(-2\sin(\pi \alpha_j)\bigr) \cdot \Theta_{\gamma}.
\end{equation*}
 It remains to verify that
\begin{equation} \label{PullOutIntegral}
  \pf(X) = \int_{\T^N} \pf \bigl(X(\tau)\bigr) \dd\tau.
\end{equation}
To this end, observe that
\begin{equation*}
  \langle E_j^\prime, \theta_a \rangle_{L^2} = \int_\T \bigl\langle E_j^\prime(\tau_a), \theta_a(\tau_a)\bigr\rangle \dd \tau_a 
\end{equation*}
and
\begin{equation*}
  \langle \theta_a, G \theta_b \rangle_{L^2} = \int_{\T^2_\circ} \bigl\langle \theta_a(\tau_a), g(\tau_a, \tau_b) \theta(\tau_b) \bigr\rangle \dd \tau_a \dd \tau_b.
\end{equation*}
Now since the Pfaffian is linear in each entry, we can ``pull out'' the integrals, and \eqref{PullOutIntegral} follows.
\end{proof}

We now collected all necessary preliminaries in order to prove our main theorem.

\begin{proof}[of Thm.~\ref{ThmMainTheoremLvsPf}]
By \eqref{qAndD} and \eqref{TopandD}, it suffices to compare the densities $\ber$ and $\ber^\prime$. Explicitly, we only need to show that for all $\tau = (\tau_1, \dots, \tau_N) \in \T_\circ^N$ and all $\xi_a \in T_{\gamma(\tau_a)}' X$, we have
\begin{equation} \label{EqualityOfDensities}
  \Phi\bigl(\ber(\tau)[\xi_1, \dots, \xi_N]\bigr) = \ber^\prime(\tau)[\xi_1, \dots, \xi_N].
\end{equation}
By skew-symmetry of both functionals, the densities $\ber$ and $\ber^\prime$ are invariant under the signed action of $S_N$, hence in fact, it suffices to verify \eqref{EqualityOfDensities} for all $\tau \in \T_\circ^N$ with $\tau_1 < \dots < \tau_N$, which is what we will do.

In order to apply Lemma~\ref{PfDensity}, choose an orthonormal basis $e_1, \dots, e_n$ of $T_{\gamma(0)}X$ such that the parallel transport $[\gamma\|_1^0]^{TX}$ takes the form \eqref{MatrixOfParallelTransport} for numbers $\alpha_1, \dots, \alpha_m \notin \Z$, let $e_1^\prime, \dots, e_n^\prime$ be the corresponding dual basis, and let $E_1, \dots, E_n$, respectively $E_1^\prime, \dots, E_n^\prime$ be the sections of $\gamma^*TX$, respectively $\gamma^*T^\prime X$ obtained by parallel translation, as in \eqref{ParallelTranslatesEj}.
With a view on the definition \eqref{DefinitionPhiGamma} of $\Phi_\gamma$ and the formula \eqref{Formulaberprime} for $\ber^\prime$, we are left to
verify  the equation
\begin{equation} \label{ScalarEquation}
\begin{aligned}
\ber_P(\tau)[\xi_1, \dots, \xi_N] = ~& (-1)^{(N-n)/2} \cdot \epsilon_0 \cdot  \sign(e_1, \dots, e_n)   \prod_{j=1}^m \bigl( - 2 \sin(\alpha_j)\bigr)  \pf\bigl(X(\tau)\bigr),
\end{aligned}
\end{equation}
for some spin structure $P$ on $\gamma^* TX$, where $X(\tau)$ is the matrix from \eqref{XMatrixInLemma} and $\epsilon_0 \in \{\pm 1\}$ is the sign from formula \eqref{Eq:FormulaSpinorParallelTransport}. Let us choose $e_1, \dots, e_n$ positively oriented so that $\sign(e_1, \dots, e_n) = 1$. 

Let us take a look at the left hand side of \eqref{ScalarEquation}. From the fact that the Levi-Civita connection on the spinor bundle satisfies a product rule with respect to Clifford multiplication, it follows that
the parallel transports in $T^\prime X$ and $\Sigma_P$ are compatible in the sense that for any $\xi \in T_{\gamma(t)} X$, we have
\begin{equation*}
  \cc\bigl(\xi\bigr) = [\gamma\|_1^t]^{\Sigma_P}\cc([\gamma\|_t^1]^{T^\prime X}\xi) [\gamma\|_t^1]^{\Sigma_P}.
\end{equation*}
 Therefore, for $0 \leq \tau_1 < \dots < \tau_N$ and $\xi_a \in T_{\gamma(\tau_a)}X$, we have
\begin{equation} \label{ParalleltransportCalc}
\ber_P(\tau_1, \dots, \tau_N)[\xi_1, \dots, \xi_N] = (-1)^N  2^{-N/2}\str \Bigl( [\gamma\|_1^0]^{\Sigma_P} \cc(\xi_1^{\|}) \cdots \cc(\xi_N^{\|})\Bigr),
\end{equation}
where 
\begin{equation} \label{ParalleltranslatesXi}
\xi_a^{\|} = [\gamma\|_{\tau_a}^1]^{T^\prime X} \xi_a \in T_{\gamma(0)}'X
\end{equation}
are the elements of $T_{\gamma(0)}X$ obtained by parallel translating $\xi_a$. 

By multi-linearity, it suffices to verify \eqref{ScalarEquation} in the case that for each $a$, we have $\xi_a^{\|} = e_{j_a}^\prime$ for some index $1\leq j_a \leq n$.
We will now show that under this assumption, both sides of \eqref{ScalarEquation} can be written as a product, with factors corresponding to the invariant subspaces of the parallel transport around $\gamma$. These are the subspaces $V_j \subset T_{\gamma(0)} X$, $j=0, 1, \dots, m$ given by
\begin{equation*}
  V_0 = \mathrm{span}\{e_{2m+1}, \dots, e_n\}, \qquad \text{and} \qquad V_j = \mathrm{span}\{e_{2j-1}, e_{2j}\}, \quad j=1, \dots, m.
\end{equation*}
We then have the tensor product factorization 
\begin{equation*}
\Cl(T_{\gamma(0)}X) \cong \Cl(V_0)\otimes \Cl(V_1) \otimes \cdots \otimes \Cl(V_m)
\end{equation*}
of Clifford algebras, and it follows from the formula \eqref{DefinitionSupertrace} for the supertrace that if an element $a$ of this Clifford algebra is written as $a = a_0 \otimes a_1 \otimes \cdots \otimes a_m$ with respect to this decomposition, then its supertrace factors as 
\begin{equation*}
  \str(a) = \str_{V_0}(a_0) \str_{V_1}(a_1) \cdots \str_{V_m}(a_m),
\end{equation*}
where $\str_{V_j}$ is the supertrace of the Clifford algebra $\Cl(V_j)$. 

For $j=0, \dots, m$, let $I_j \subset \{1, \dots, N\}$ be the set of indices $a$ such that $\xi_a^{\|} \in V_j$. Let $\sigma \in S_N$ be the unique permutation such that for $a < b$, we have $\sigma_a > \sigma_b$ if and only if $a \in I_j$, but $b \in I_i$ for some $i < j$. Then by the formula \eqref{Eq:FormulaSpinorParallelTransport} for the parallel transport, we have
\begin{equation*}
[\gamma\|_1^0]^{\Sigma_P} \cc(\xi_1^{\|}) \cdots \cc(\xi_N^{\|}) = \epsilon_0 \, \sgn(\sigma) \cdot \underbrace{\prod_{a \in I_0} \cc(\xi_a^{\|})}_{\in \Cl(V_0)} \cdot \prod_{j=1}^m \underbrace{\bigl(\cos(\pi\alpha_j) + \sin(\pi\alpha_j)\cc_{2j-1}\cc_{2j}\bigr) \prod_{a \in I_j} \cc(\xi_a^{\|})}_{\in \Cl(V_j)}.
\end{equation*}
Summing up, we obtain
\begin{equation} \label{FactorizationF}
\ber_P(\tau_1, \dots, \tau_N)[\xi_1, \dots, \xi_N] = \epsilon_0 \cdot \sgn(\sigma) \cdot \ber^{(0)} \cdot \ber^{(1)}\cdots \ber^{(m)},
\end{equation}
where $\ber^{(j)}$ is $2^{-|I_j|/2}$ times the supertrace of the piece in $\Cl(V_j)$ from the previous formula (this is written out explicitly in \eqref{FormulaFj} below).

Regarding the right hand side of \eqref{ScalarEquation}, first observe that that the parallel translates of the vector spaces $V_j$ form pairwise orthogonal vector bundles $\mathcal{V}_j$ over $\T$ that are preserved by $\nabla_{\dot{\gamma}}$. Therefore the Green's function also maps sections of $\mathcal{V}_j$ to sections of $\mathcal{V}_j$, which entails that $\langle \xi_a, g(\tau_a, \tau_b) \xi_b\rangle = 0$ whenever $\xi_a^{\|}$ and $\xi_b^{\|}$ are contained in different $V_j$. 
Similarly, we have 
$\langle E_j^\prime(\tau_a), \xi_a \rangle = \langle e_j^\prime, \xi_a^{\|} \rangle = 0$
 unless $\xi_a^{\|} \in V_0$.
 By these observations, after applying the permutation $\sigma \in S_N$ to the indices of the elements $\xi_a$ appearing in the matrix $X(\tau)$, the resulting matrix has block diagonal form, with blocks $X^{(0)}$, $X^{(1)}$, \dots, $X^{(m)}$. Here the matrix $X^{(0)}$ can be described as the matrix obtained from $X(\tau)$ by deleting all rows and columns containing a $\xi_a$ with $a \notin I_0$, while the matrices $X^{(j)}$ are obtained from $X$ by deleting the first $d$ rows and columns, as well as all rows and columns containing a $\xi_a$ with $a \notin I_j$.
 Applying the permutation formula \eqref{PfaffianPermutation}, we therefore obtain
\begin{equation} \label{FactorizationPfaffian}
  \pf\bigl(X(\tau)\bigr) = \sgn(\sigma) \cdot \pf(X^{(0)}) \cdot \pf(X^{(1)}) \cdots \pf(X^{(m)}).
\end{equation}
Plugging \eqref{FactorizationF} and \eqref{FactorizationPfaffian} into \eqref{ScalarEquation} we are left to show the identities
\begin{align}\label{Identity0}
  \ber^{(0)} &= \pf(X^{(0)}),  \\
  \ber^{(j)} &= 2\sin(\alpha_j) \cdot \pf(X^{(j)}), \quad j=1, \dots, m. \label{Identityj}
\end{align}

\medskip

{\em Case $j \in \{1, \dots, m\}$}: Explicitly, the formula derived above for $\ber^{(j)}$ is
\begin{equation} \label{FormulaFj}
\ber^{(j)} =  2^{-|I_j|/2}\, \str_{V_j} \Bigl(\bigl(\cos(\pi \alpha_j) + \sin(\pi \alpha_j) \cc_{2j-1} \cc_{2j}\bigr) \prod_{a \in I_j} \cc(\xi_{a}^{\|})\Bigr).
\end{equation}
We  observe that both sides of \eqref{Identityj} are  zero unless $|I_j|$ is even: $\ber_j$ is because $V_j$ is two-dimensional and hence the supertrace $\str_{V_j}$ is an even functional; $\pf(X^{(j)})$ is, because the matrix $X^{(j)}$ has dimensions  $|I_j|\times |I_j|$ and the Pfaffian of an odd-dimensional matrix is zero. 

We first calculate $\ber^{(j)}$, assuming that $|I_j|$ is even.  By assumption, for each $a \in I_j$, we have $\xi_{a}^{\|} = e_{2j-1}^\prime$ or $e_{2j}^\prime$; let us say that $e_{2j-1}^\prime$ appears $k$ times this way and that $e_{2j}^\prime$ appears $|I_j|-k$ times. Let $a_1 < \dots < a_{|I_j|}$ run through the elements of $I_j$ and let $\rho$ be the permutation of $I_j$ such that 
\begin{equation*}
  \xi_{\rho_{a_1}}^{\|} = \dots = \xi_{\rho_{a_k}}^{\|} = e_{2j-1}^\prime, \qquad  \xi_{\rho_{a_{k+1}}}^{\|} = \dots = \xi_{\rho_{a_{|I_j|}}}^{\|} = e_{2j}^\prime,
\end{equation*}
and such that for $\rho_{a_1} < \dots < \rho_{a_k}$ and $\rho_{a_{k+1}} < \dots < \rho_{a_{|I_j|}}$.
Then by the choice of $\rho$, 
\begin{equation*}
  \cc(\xi_{a_1}^{\|}) \cdots \cc(\xi_{a_{|I_j|}}^{\|}) = \sgn(\rho)\cdot \cc(\xi_{\rho_{a_1}}^{\|}) \cdots \cc(\xi_{\rho_{a_{|I_j|}}}^{\|})= \sgn(\rho) \cdot \cc_{2j-1}^k \cc_{2j}^{|I_j|-k}.
\end{equation*}
By the formula \eqref{DefinitionSupertrace} for the supertrace, we therefore find 
\begin{equation} \label{ResultFj}
\begin{aligned}
\ber^{(j)} &= \sgn(\rho) \cdot  2^{-|I_j|/2} \cdot \str_{V_j}\Bigl(\bigl(\cos(\pi \alpha_j) + \sin(\pi \alpha_j) \cc_{2j-1} \cc_{2j}\bigr) \cc_{2j-1}^k \cc_{2j}^{|I_j|-k}\Bigr) \\
&= \sgn(\rho) \cdot (-2)^{-|I_j|/2} \cdot \begin{cases} 2 \sin(\pi \alpha) & k~\text{even} \\ -2 \cos(\pi \alpha) & k~\text{odd}.\end{cases}
\end{aligned} 
\end{equation}

We now calculate $\pf(X^{(j)})$, still assuming that $|I_j|$ is even.
Since the parallel transport $[\gamma\|_1^0]^{T^\prime X}$ in $\gamma^*T^\prime X$ is given by the matrix \eqref{MatrixOfParallelTransport} with respect to the dual basis $e_1^\prime, \dots, e_n^\prime$, its restriction to $V_j^\prime \subset T^\prime_{\gamma(0)} X$ is given by
\begin{equation*}
  R = \begin{pmatrix} \cos(2\pi \alpha_j) & -\sin(2\pi \alpha_j) \\ \sin(2 \pi \alpha_j) & \cos(2\pi \alpha_j) \end{pmatrix}.
\end{equation*}
Now, smooth sections of the subbundle $\V_j^\prime \subset \gamma^*T^\prime X$ can be identified with smooth functions $\theta: \R \rightarrow \R^2$ satisfying the quasi-periodic boundary condition
\begin{equation} \label{BoundaryConditionAlpha}
  \theta(t)  = R\, \theta(t+1).
\end{equation}
We claim that the Green's operator and Green's function to the problem $\theta^\prime = 0$ with respect to these boundary conditions are given by
\begin{equation*}
  (G \theta)(t)  = \int_0^1 g(t, s) \theta(s) \dd s, \qquad 
  g(t, s) = H(t-s)\mathbf{1} + (R^*-\mathbf{1})^{-1},
\end{equation*}
where $H$ is the Heaviside step function. To see this, it suffices to check that $(G\theta)^\prime = \theta$ and (for $\theta$ satisfying the boundary conditions) $G\theta^\prime = \theta$, and that for any $\theta$, $G\theta$ satisfies the boundary condition \eqref{BoundaryConditionAlpha}; we leave this to the reader.
We have  
\begin{equation*}
(R^*-\mathbf{1})^{-1} = -\frac{1}{2} \begin{pmatrix} 1 & \cot(\pi \alpha) \\ -\cot(\pi \alpha) & 1 \end{pmatrix}.
\end{equation*}
Therefore, we have
\begin{equation*}
\begin{aligned}
   \bigl\langle \xi_a, g(\tau_a, \tau_b)\xi_b &\bigr\rangle  &= -\frac{1}{2} \times
   \begin{cases}
    1 & \xi_a^{\|} = \xi_b^{\|} = e_{2j-1}^\prime~\text{and}~ \tau_a < \tau_b,\\
    -1 & \xi_a^{\|} = \xi_b^{\|} = e_{2j-1}^\prime~\text{and}~ \tau_a > \tau_b,\\
    1 & \xi_a^{\|} = \xi_b^{\|} = e_{2j}^\prime~\text{and}~  \tau_a < \tau_b ,\\
    -1 & \xi_a^{\|} = \xi_b^{\|} = e_{2j}^\prime~\text{and}~ \tau_a > \tau_b,\\
      \cot(\pi \alpha) & \xi_a^{\|} = e_{2j-1}^\prime~\text{and}~ \xi_b^{\|} = e_{2j}^\prime, \\
     -\cot(\pi \alpha) & \xi_a^{\|} = e_{2j}^\prime~\text{and}~ \xi_b^{\|} = e_{2j-1}^\prime.
   \end{cases}
\end{aligned}
\end{equation*}
The matrix $X^{(j)}$ has the form \eqref{MatrixNoKernel}, with $\xi_a$ running through $a \in I_j$. Now let $\tilde{X}^{(j)}$ be the matrix obtained from $X^{(j)}$ by permuting the rows and columns using the permutation $\rho$. By choice of $\rho$, we have $\tau_{\rho_{a_1}} < \dots < \tau_{\rho_{a_k}}$ and $\tau_{\rho_{a_{k+1}}} < \dots < \tau_{\rho_{a_{|I_j|}}}$, hence with a view on \eqref{MatrixNoKernel}, we have
\begin{equation*}
  \tilde{X}^{(j)} = -\frac{1}{2} 
  \begin{pmatrix} 
  0 & 1 & \cdots & 1 & \vline & \cot(\pi \alpha_j) & \cdots & \cdots & \cot(\pi \alpha_j) \\
   -1  & \ddots & \ddots &\vdots & \vline & \vdots & & & \vdots \\
   \vdots  &  \ddots   & \ddots & 1 & \vline &  \vdots & & & \vdots \\
   -1  &  \dots    &  -1         & 0 & \vline & \cot(\pi \alpha_j) & \cdots & \cdots & \cot(\pi \alpha_j) \\
     &      &           &  & \vline &  &  &  &  \\
   \hline & & & & \vline & & & & \\
 -\cot(\pi \alpha_j)  & \dots &  \dots  & -\cot(\pi \alpha_j) & \vline & 0 & 1 & \cdots & 1\\
   \vdots& & & \vdots& \vline &  -1  & \ddots & \ddots & \vdots \\
  \vdots & & & \vdots& \vline &  \vdots  &  \ddots          & \ddots & 1 \\
 -\cot(\pi \alpha_j)  &\dots  &\dots & -\cot(\pi \alpha_j)& \vline & -1   & \dots           &  -1          & 0
  \end{pmatrix},
\end{equation*}
where the upper left block has dimension $k \times k$. 

\begin{lemma} \label{LemmaPfaffian1}
Let $0 \leq k \leq M$ be integers with $M$ even and $c \in \R$. Consider the $M \times M$-matrix
\begin{equation*}
S_{k, M}(c) =   \begin{pmatrix} 
  0 & 1 & \cdots & 1 & \vline & c & \cdots & \cdots & c \\
   -1  & \ddots & \ddots &\vdots & \vline & \vdots & & & \vdots \\
   \vdots  &  \ddots   & \ddots & 1 & \vline &  \vdots & & & \vdots \\
   -1  &  \dots    &  -1         & 0 & \vline & c & \cdots & \cdots & c \\
   \hline
 -c  & \dots &  \dots  & -c & \vline & 0 & 1 & \cdots & 1\\
   \vdots& & & \vdots& \vline &  -1  & \ddots & \ddots & \vdots \\
  \vdots & & & \vdots& \vline &  \vdots  &  \ddots          & \ddots & 1 \\
 -c  &\dots  &\dots & -c& \vline & -1   & \dots           &  -1          & 0
  \end{pmatrix},
\end{equation*}
where the upper left block has dimension $k \times k$. Then 
\begin{equation} \label{PfaffianLemmaFormula}
\pf\bigl(S_{k, M}(c)\bigr) = \begin{cases} 1 & k ~\text{even}  \\ c & k~\text{odd}. \end{cases}
\end{equation}
\end{lemma}

\begin{proof}
It is straightforward to check the proposition in the case that $M \leq 2$ and for any $M$ in the case that $k=0$ or $M$.
The result then follows from induction on $M$, using the development formula \eqref{PfaffianRecursion}. Namely, assume that we know the claim for $M-2$ and any $0 \leq k \leq M-2$. Then we have
\begin{equation*}
\pf\bigl(S_{k, M}(c)\bigr) = \sum_{a=2}^k (-1)^a \cdot  \pf\bigl(S_{k-2, M-2}(c)\bigr) + \sum_{a=k+1}^M (-1)^a c \cdot \pf\bigl(S_{k-1, M-2}(c)\bigr).
\end{equation*}
Assume \eqref{PfaffianLemmaFormula} for $M-2$ and any $k$. Then in the case that $k$ is odd, we get
\begin{equation*}
  \pf\bigl(S_{k, M}(c)\bigr) = \underbrace{\sum_{a=2}^k (-1)^a}_{=0} \cdot  \,c + \underbrace{\sum_{a=k+1}^M (-1)^a}_{=1} \cdot \,c = c,
\end{equation*}
while if $k$ is even,
\begin{equation*}
\pf\bigl(S_{k, M}(c)\bigr) = \underbrace{\sum_{a=2}^k (-1)^a}_{=1} + \underbrace{\sum_{a=k+1}^M (-1)^a}_{=0} \cdot \,c^2 = 1.
\end{equation*}
This finishes the proof.
\end{proof}

We observe that $\tilde{X}^{(j)} = -\frac{1}{2}S_{k, |I_j|}(\cot(\pi\alpha_j))$. Applying \eqref{PfaffianPermutation} and Lemma~\ref{LemmaPfaffian1} therefore yields
\begin{equation*}
\begin{aligned}
2 \sin(\pi \alpha_j) \cdot \pf(X^{(j)}) &= \sgn(\rho) \cdot 2 \sin(\pi \alpha_j) \cdot (-2)^{-|I_j|/2}\cdot \pf\bigl(S_{k, |I_j|}(\cot(\pi\alpha_j)\bigr)\\
&=  \sgn(\rho) \cdot (-2)^{-|I_j|/2} \times \begin{cases} 2 \sin(\pi\alpha_j) & k ~\text{even}  \\ 
2 \cos(\pi \alpha) & k~\text{odd}. \end{cases}
\end{aligned}
\end{equation*}
This equals $\ber^{(j)}$, as calculated in \eqref{ResultFj}, and hence verifies \eqref{Identityj}.

\medskip

{\em Case $j=0$:} It is left to verify \eqref{Identity0}. In the case that $n = 2m$, both sides of \eqref{Identity0} are equal to one by convention and nothing is to check. In the case that $d=n-2m > 0$, it is convenient to make another simplification. Namely, observe that the vector bundle $\V_0^\prime \subset \gamma^* T^\prime X$ splits up into $d=n-2m$ subbundles invariant under parallel transport, corresponding to the vectors $e_{2m+1}^\prime, \dots, e_n^\prime$. Therefore, similarly to the arguments above, we can decompose 
\begin{equation*}
\ber^{(0)} = \pm \ber^{(0)}_{2m+1} \cdots \ber^{(0)}_n
\end{equation*}
where the term $\ber^{(0)}_{k}$ corresponds to those indices $a$ with $\xi_a^{\|} = e_k^\prime$ and the sign is that of a suitable permutation. A similar decomposition can be made for the right hand side of \eqref{Identity0}, with the same sign in front. Comparing the terms individually, this has the effect of reducing to the case that $d=1$, which we assume subsequently in order to simplify notation.

Let $a_1 < \dots < a_{|I_0|}$ run through the elements of $I_0$. Then under the above simplifying assumptions,
\begin{equation*}
  \ber^{(0)} 
  = 2^{-|I_0|/2} \str_{V_0}\Bigl( \cc(\xi_{a_1}^{\|})\cdots \cc(\xi_{a_{|I_0|}}^{\|}) \Bigr) 
  = 2^{-|I_0|/2} \str_{V_0}(\cc_n^{|I_0|}),
\end{equation*}
where $\str_{V_0}$ is the supertrace of the Clifford algebra $\Cl(V_0) = \Cl(\mathrm{span}\{e_n\})$.
Since $V_0$ is one-dimensional, the supertrace is odd, hence $\ber^{(0)}$ is non-zero only if $|I_0|$ is odd, and in this case,
\begin{equation} \label{ResultF0}
  \ber^{(0)} = 2^{-|I_0|/2} (-1)^{(|I_0|-1)/2} \str(\cc_n) = (-2)^{-(|I_0|-1)/2}.
\end{equation}
This has to be compared to $\pf(X^{(0)})$. We will use the following lemma.

\begin{lemma} \label{LemmaPfaffian2}
For $M \in \N$ and numbers $\tau_1, \dots, \tau_M$, define the $M \times M$ matrix
\begin{equation*} 
S(\tau_1, \dots, \tau_M) = \begin{pmatrix}
  0 &  \tau_1 - \tau_2 +\tfrac{1}{2} & \cdots  &  \tau_1 - \tau_M + \tfrac{1}{2} \\
\tau_2 - \tau_1 - \tfrac{1}{2} & 0 & \ddots   & \vdots \\
\vdots  &   \ddots& \ddots & \tau_{M-1} - \tau_{M} +\tfrac{1}{2} \\
\tau_M - \tau_{1}  - \tfrac{1}{2} & \dots & \tau_{M} - \tau_{M-1} - \tfrac{1}{2}& 0
\end{pmatrix}.
\end{equation*}
Then if $M$ is even,
\begin{equation}  \label{EqLemmaEven}
  \pf\bigl(S(\tau_1, \dots, \tau_M)\bigr) = 2^{1 - M/2} \sum_{a=1}^M (-1)^{a-1} \tau_a + 2^{-M/2},
\end{equation}
while if $M$ is odd, then
\begin{equation} \label{EqLemmaOdd}
  \sum_{a=1}^M (-1)^{a+1} \pf\bigl(S(\tau_1, \dots, \hat{\tau}_a, \dots \tau_M)\bigr) = 2^{-(M-1)/2}.
\end{equation}
\end{lemma}

\begin{proof}
We consider the case that $M$ is even. The case $M=2$ is clear with a view on \eqref{PfaffianTwoByTwoMatrix}. The case $M \geq 4$ even follows by induction, using the recursion \eqref{PfaffianRecursion}. Namely, assume that \eqref{EqLemmaEven} is known for $M-2$. Then 
\begin{equation*}
\begin{aligned}
\pf\bigl(S(\tau_1, \dots, \tau_M)\bigr) &= \sum_{a=2}^M (-1)^a (\tau_1 - \tau_a + \tfrac{1}{2}) \pf(\tau_2, \dots, \hat{\tau}_a, \dots, \tau_N) \\
&= 2^{2-M/2}\sum_{a=2}^M (-1)^a (\tau_1 - \tau_a + \tfrac{1}{2}) \left(\sum_{b=2}^{a-1} (-1)^{b} \tau_b + \sum_{b=a+1}^M (-1)^{b-1} \tau_b + \tfrac{1}{2}\right).
\end{aligned}
\end{equation*}
The right hand side is a polynomial of degree two in $\tau_1, \dots, \tau_M$, the linear term of which is $2^{-M/2}$, as claimed. The quadratic part of this polynomial is $2^{2-M/2}$ times
\begin{equation*}
\begin{aligned}
  \sum_{2 \leq b < a \leq M} (-1)^{a+b} (\tau_1 - \tau_a)\tau_b + \sum_{2 \leq a < b \leq M} (-1)^{a + b-1} (\tau_1 - \tau_a)\tau_b.
\end{aligned}
\end{equation*}
Observe that the parts of the sums not containing $\tau_1$ cancel. We are left with $\tau_1$ times
\begin{equation*}
\begin{aligned}
  \sum_{2 \leq b < a \leq M} (-1)^{a+b} \tau_b + \sum_{2 \leq a < b \leq M} (-1)^{a + b-1} \tau_b
  = \sum_{b=2}^M (-1)^b \tau_b \cdot \left( \sum_{a = b+1}^M (-1)^a + \sum_{a=2}^{b-1} (-1)^{a-1} \right).
\end{aligned}
\end{equation*}
Now since $M$ is even, the two sums on the right always cancel, for any $b$. We conclude that the quadrating terms of $\pf(S(\tau_1, \dots, \tau_M))$ are zero. The linear terms are $2^{1-M/2}$ times
\begin{equation*}
\begin{aligned}
\sum_{a=1}^M& (-1)^a (\tau_1 - \tau_a) + \sum_{2 \leq b \leq a \leq M} (-1)^{a+b} \tau_b + \sum_{2 \leq a \leq b \leq M} (-1)^{a+b} (-1)^{a+b+1}\tau_b  \\
&= \sum_{a=1}^M (-1)^a \tau_1 + \sum_{a=1}^M (-1)^{a-1} \tau_a  + \sum_{b=2}^M (-1)^b \tau_b \cdot \left( \sum_{a = b+1}^M (-1)^a + \sum_{a=2}^{b-1} (-1)^{a-1} \right)
\end{aligned}
\end{equation*}
The first sum vanishes because $M$ is even, and the third sum vanishes as before as the two sums in the bracket cancel out for each $b$. This shows that $\pf(S(\tau_1, \dots, \tau_M))$ also has the desired linear term and finishes the proof of \eqref{EqLemmaEven}.

To prove the second formula, observe that by the previous result, the left hand side of \eqref{EqLemmaOdd} is a polynomial of degree one in the $\tau_a$. Its linear term is
\begin{equation*}
 2^{-(M-1)/2} \sum_{a=1}^M (-1)^{a+1}  = 2^{-(M-1)/2}
\end{equation*}
as claimed, where we used that $M$ is odd. The linear terms are $2^{1-M/2}$ times
\begin{equation*}
\sum_{a=1}^M (-1)^{a+1} \left(\sum_{b=1}^{a-1} (-1)^{b-1} \tau_b + \sum_{b=a+1}^{M} (-1)^{b} \tau_b\right)  = 
\sum_{b=1}^M (-1)^b\tau_b \left( \sum_{a=b+1}^M (-1)^{a} + \sum_{a=1}^{b-1} (-1)^{a+1} \right).
\end{equation*} 
Since $M$ is odd, the term in the brackets on the right hand side vanishes, which finishes the proof.
\end{proof}

Now, smooth sections of the bundle $\mathcal{V}_0^\prime \subset \gamma^*T^\prime X$ over $\T$ can be identified with smooth functions $\theta : \R \rightarrow \R$ such that $\theta(t+1) = \theta(t)$. It is straightforward to check that the Green's function to this problem is given by
\begin{equation*}
 (G\theta)(t) = \int_0^1 g(t, s) \theta(s) \dd s, \qquad 
 g(t, s) = H(t - s) + s - t - \tfrac{1}{2}.
\end{equation*}
Since under our assumptions, we have $\tau_a < \tau_b$ if $a < b$, we have for  $a, b \in I_0$ that
\begin{equation*}
  \bigl\langle \theta(\tau_a), G(\tau_a, \tau_b)\theta(\tau_b)\bigr\rangle = 
  \begin{cases} 
  \tau_b - \tau_a - \tfrac{1}{2} & \text{if}~~ a < b \\
  \tau_b - \tau_a + \tfrac{1}{2}   & \text{if}~~ a > b.
  \end{cases}
\end{equation*}
Moreover, $\langle e_n^\prime, \xi_a^{\|}\rangle = \langle e_n^\prime, e_n^\prime\rangle = 1$ for all $a \in I_0$. The matrix $X^{(0)}$ is therefore 
\begin{equation*}
X^{(0)} = 
\begin{pmatrix} 
0 & 1 & \cdots & 1 \\ 
-1 & & & \\
\vdots & & S(\tau_{a_1}, \dots, \tau_{a_{|I_0|}}) & \\
-1 & & &
\end{pmatrix},
\end{equation*}
where $S(\tau_{a_1}, \dots, \tau_{a_{|I_0|}})$ is the matrix from Lemma~\ref{LemmaPfaffian2}. 
Using the development formula \eqref{PfaffianRecursion} and applying \eqref{EqLemmaOdd}, we therefore obtain
\begin{equation} \label{MatrixXV0}
\begin{aligned}
(-1)^{(|I_0|-1)/2} \cdot \pf(X^{(0)}) &=  \sum_{i=1}^{|I_0|} (-1)^{i+1} \pf\bigl(S(\tau_{a_1}, \dots, \hat{\tau}_{a_i}, \dots, \tau_{a_{|I_0|}})\bigr) \\
&=    (-2)^{-(|I_0|-1)/2}.
\end{aligned}
\end{equation}
This indeed coincides with \eqref{ResultF0} and shows \eqref{Identity0}.

\medskip

We have now verified \eqref{Identity0} and \eqref{Identityj}, so the proof is complete.
\end{proof}


\bibliography{LiteraturLoopSpaces}

\end{document}